\documentclass[11pt]{amsart}
\usepackage{fullpage}
\usepackage{booktabs}

\usepackage{amssymb,amsthm}
\usepackage[vcentermath]{youngtab}
\usepackage{graphicx,epsfig}
\usepackage{pst-plot,pstricks}

\usepackage{color}

\makeatletter
\newfont{\footsc}{cmcsc10 at 8truept}
\newfont{\footbf}{cmbx10 at 8truept}
\newfont{\footrm}{cmr10 at 10truept}
\renewcommand{\ps@plain}{%
\renewcommand{\@oddfoot}{\footsc 
\hfil\footrm\thepage}}
\makeatother
\definecolor{mygreen}{rgb}{0.286,0.463,0.329}
\definecolor{mygrey}{rgb}{0.7,0.7,0.7} 
\definecolor{greygreen}{rgb}{0.357,0.51,0.447}
\definecolor{greyblue}{rgb}{0.357,0.447,0.51}
\definecolor{ltgreyblue}{rgb}{0.557,0.647,0.71}
\definecolor{titlecol}{rgb}{0.326,0.082,0.549} 
\definecolor{myhdrcol}{rgb}{0.326,0.082,0.549}
\definecolor{hlit}{rgb}{0,0.6,0}
\definecolor{medcol}{rgb}{0.8,0.8,0.8}
\definecolor{myorange}{rgb}{1,0.451,0}
\definecolor{myred}{rgb}{0.326,0.082,0.549}
\definecolor{myemph}{rgb}{1,0.537,0.184}

\theoremstyle{plain}
\newtheorem{theorem}{Theorem}

\newtheorem{corollary}[theorem]{Corollary}

\theoremstyle{definition}

\newtheorem{example}[theorem]{Example}
\newtheorem{remark}[theorem]{Remark}

\newcommand{\Mtrans}{\mathcal{M}} 
\newcommand{\N}{\mathbb{N}}
\newcommand{\Q}{\mathbb{Q}}

\newcommand{\symmgp}{\mathfrak{S}}
\newcommand{\FQ}{F} 
\DeclareMathOperator{\Asc}{Asc}
\DeclareMathOperator{\Bre}{Bre}
\DeclareMathOperator{\chg}{chg}
\DeclareMathOperator{\comp}{comp}
\DeclareMathOperator{\Comp}{Comp}

\DeclareMathOperator{\Des}{Des}

\DeclareMathOperator{\espec}{Esp}
\DeclareMathOperator{\id}{id}
\DeclareMathOperator{\IDes}{IDes}

\DeclareMathOperator{\NSym}{NSym}
\DeclareMathOperator{\Par}{Par}
\DeclareMathOperator{\QAsc}{QAsc}

\DeclareMathOperator{\sgn}{sgn}
\DeclareMathOperator{\sort}{sort}
\DeclareMathOperator{\spec}{Sp}
\DeclareMathOperator{\QSp}{QSp}
\DeclareMathOperator{\QSym}{QSym}

\DeclareMathOperator{\ssyt}{SSYT}

\DeclareMathOperator{\stdz}{stdz}
\DeclareMathOperator{\sub}{sub}
\DeclareMathOperator{\Sym}{Sym}
\DeclareMathOperator{\syt}{SYT}
\DeclareMathOperator{\tstat}{tstat}

\DeclareMathOperator{\wt}{wt}
\DeclareMathOperator{\comaj}{comaj}
\DeclareMathOperator{\cont}{cont}
\DeclareMathOperator{\inv}{inv}
\DeclareMathOperator{\SCT}{SCT}

\DeclareMathOperator{\SSCT}{SSCT}
\DeclareMathOperator{\SumAsc}{SumAsc}
\newcommand{\field}{\mathbb{F}}

\newcommand{\ten}{$10$}
\newcommand{\eleven}{$11$}
\newcommand{\twelve}{$12$}
\newcommand{\thirteen}{$13$}
\newcommand{\fourteen}{$14$}
\newcommand{\fifteen}{$15$}
\newcommand{\sixteen}{$16$}
\newcommand{\seventeen}{$17$}
\newcommand{\eighteen}{$18$}

\newlength{\cellsize}
\newcommand\tableau[1]{
\vcenter{
\let\\=\cr
\baselineskip=-16000pt
\lineskiplimit=16000pt
\lineskip=0pt
\halign{&\tableaucell{##}\cr#1\crcr}}}

\newcommand{\tableaucell}[1]{{%
\def \arg{#1}\def \void{}%
\ifx \void \arg
\vbox to \cellsize{\vfil \hrule width \cellsize height 0pt}%
\else
\unitlength=\cellsize
\begin{picture}(1,1)
\put(0,0){\makebox(1,1){$#1$}}
\put(0,0){\line(1,0){1}}
\put(0,1){\line(1,0){1}}
\put(0,0){\line(0,1){1}}
\put(1,0){\line(0,1){1}}
\end{picture}%
\fi}}

\begin{document} 


\title{Transition matrices for symmetric and quasisymmetric
  Hall-Littlewood polynomials}


\date{\today}

\author{Nicholas A. Loehr}
\address{Dept. of Mathematics\\
  Virginia Tech \\
  Blacksburg, VA 24061-0123 \\
and Mathematics Department\\
United States Naval Academy \\
Annapolis, MD 21402-5002}
\email{nloehr@vt.edu}

\author{Luis G. Serrano}
\address{Laboratoire de combinatoire et d'informatique math\'ematique (LaCIM)\\
  Universit\'e du Qu\'ebec \`a Montr\'eal\\
  Montr\'eal, QC, Canada}
\email{serrano@lacim.ca}

\author{Gregory S. Warrington}
\address{Dept. of Mathematics and Statistics\\
  University of Vermont \\
  Burlington, VT 05401}
\email{gregory.warrington@uvm.edu}

\thanks{This work was partially supported by a grant
from the Simons Foundation (\#244398 to Nicholas Loehr).}

\thanks{Second author supported by a Natural Science and Engineering
  Research Council of Canada PDF grant.}

\thanks{Third author supported in part by National Security Agency
  grant H98230-09-1-0023 and National Science Foundation grant DMS-1201312.}

\thanks{This work was partially supported by a grant from the Simons
  Foundation (\#197419 to Greg Warrington).}


\begin{abstract}
  We introduce explicit combinatorial interpretations for the
  coefficients in some of the transition matrices relating to skew
  Hall-Littlewood polynomials $P_{\lambda/\mu}(t)$ and Hivert's
  quasisymmetric Hall-Littlewood polynomials $G_\gamma(t)$.  More
  specifically, we provide: 
  \begin{enumerate}
  \item the $G$-expansions of 
     the Hall-Littlewood polynomials $P_{\lambda}(t)$, 
     the monomial quasisymmetric polynomials $M_{\alpha}$, 
     the quasisymmetric Schur polynomials $\mathcal{S}_\alpha$, 
     and the peak quasisymmetric functions $K_\alpha$;
   \item an expansion of $P_{\lambda/\mu}(t)$ in terms of the $\FQ_\alpha$'s.
  \end{enumerate}
  The $\FQ$-expansion of $P_{\lambda/\mu}(t)$ is facilitated
  by introducing \emph{starred tableaux}.
\end{abstract}

\maketitle

\section{Introduction}
\label{sec:intro}

The ring of symmetric functions $\Sym$ and the ring of quasisymmetric
functions $\QSym$ both play important roles in algebra and
combinatorics.  Much of the combinatorial richness arising from these
rings stems from their various distinguished bases and the relationships
between these bases.  The goal of this paper is to present explicit,
combinatorial descriptions of several such transition matrices
relating to the Hall-Littlewood polynomials.  Figure~\ref{fig:prism}
illustrates some of the bases discussed.

\begin{figure}[htbp]
  \centering {\scalebox{1}{\includegraphics{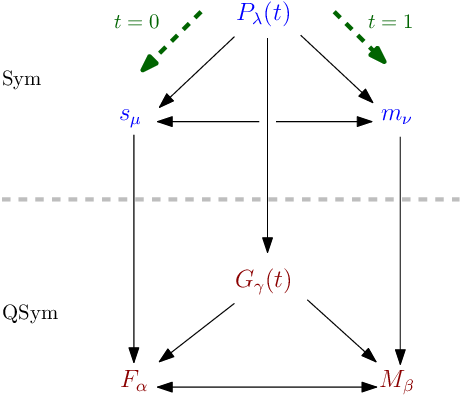}}}
  \caption{Prism of bases and transitions.}
  \label{fig:prism}
\end{figure}

In the top triangle in Figure~\ref{fig:prism} are included two classical
bases for the ring of symmetric functions: the \emph{Schur functions}
$s_\mu$ and the \emph{monomial symmetric functions} $m_\nu$.  The
$s_\mu$ and $m_\nu$ are closely related to a third, one-parameter
family of symmetric functions $P_\lambda(t)$, known as
\emph{Hall-Littlewood polynomials}.  More specifically, $P_\lambda(t)$
equals $s_\lambda$ at $t=0$, and it equals $m_\lambda$ at $t=1$.  The
$P_\lambda(t)$ arose out of a problem studied by P.~Hall.  Hall had used his
eponymous algebra (isomorphic to the algebra of symmetric functions)
to encode the structure of finite abelian $p$-groups.  However, at the
time there was no known explicit basis of symmetric functions with the
same structure constants as that of the natural basis for Hall's
algebra.  D. E. Littlewood~\cite{littlew} solved this problem in 1961
with his introduction of the $P_\lambda(t)$.

The bottom triangle of Figure~\ref{fig:prism} consists of
quasisymmetric analogues of the above bases.  In the context of
quasisymmetric functions, the \emph{monomial quasisymmetric
  functions}, $M_\beta$, are a very natural analogue of the $m_\nu$.
There are several possible quasisymmetric analogues of the Schur
functions, including the ``quasisymmetric Schur functions''
$\mathcal{S}_{\alpha}$ introduced in~\cite{hlmvw} and discussed later
in this paper.  However, for reasons described in the next paragraph,
we anchor the lower-left portion of the bottom triangle in
Figure~\ref{fig:prism} by Gessel's \emph{fundamental quasisymmetric
  functions}, denoted here by $\FQ_\alpha$.  By defining an action of 
the Hecke algebra on polynomials which leaves the quasisymmetric
functions invariant, Hivert~\cite{hivert} constructed the
\emph{quasisymmetric Hall-Littlewood polynomials}
$G_\gamma(t)$. (See also work of Lascoux, Novelli, and Thibon
\cite{LNT} for constructions of quasisymmetric and noncommutative
symmetric functions with extra parameters.) Similarly to what happens
in the top triangle, specialization of the $G_\gamma(t)$ at $t=0$ (which
corresponds to the southwest-pointing arrow in Figure~\ref{fig:prism})
yields $\FQ_\gamma$, while specialization at $t=1$ yields $M_\gamma$.

We now motivate our choice of the $\FQ_\alpha$ as the desired
quasisymmetric analogue of the Schur functions.  The Schur functions
are the prototypical example of a symmetric function with
combinatorial expansions in terms of both a collection of
\emph{semistandard objects} (i.e., semistandard Young tableaux) and of
\emph{standard objects} (i.e., standard Young tableaux). The first
case is that of the classical expansion in terms of monomials weighted
by the Kostka numbers. The second expansion (due to
Gessel~\cite{gessel}) expresses the Schur functions in terms of
fundamental quasisymmetric functions $\FQ_\alpha$.  This expansion,
which follows from the technique of \emph{standardization}, is
indicated by the vertical line connecting $s_\mu$ and $\FQ_\alpha$
in Figure~\ref{fig:prism}.  Such standardizations have been used recently
to give $F$-expansions of various symmetric functions including
plethysms of Schur functions~\cite{spleth}, the modified Macdonald
polynomials~\cite{hag,HHL}, the Lascoux-Leclerc-Thibon (LLT)
polynomials~\cite{llt}, and (conjecturally) the image of a Schur
function under the Bergeron-Garsia nabla operator~\cite{nablaschur}.

Given Hivert's construction, the following question arises. Is there
an expansion of the $P_\lambda(t)$ in terms of the $G_\gamma(t)$,
which would interpolate between the $F$-expansion of the $s_\mu$ at
$t=0$ and the $M$-expansion of the $m_\nu$ at $t=1$? The main purpose
of this paper is to provide such an expansion, as well as other
change-of-basis matrices between different bases of the Hall algebra
and the algebra of quasisymmetric functions.  In terms of
Figure~\ref{fig:prism}, we provide the middle vertical edge as well as
the reverses of the two downward directed edges in the bottom face
(namely, from each of $\FQ_\alpha$ and $M_\beta$ to $G_\gamma(t)$).  More
specifically, our principal results are:

\begin{enumerate}
\item\emph{$G$-expansion of the $P$ basis.}
In Theorem~\ref{thm:mpg} we give an explicit combinatorial expansion
of the Hall-Littlewood polynomials $P_{\lambda}(t)$ in terms of the
Hivert quasisymmetric Hall-Littlewood polynomials
$G_\gamma(t)$. This provides the desired $t$-interpolation between
Gessel's $F$-expansion of Schur polynomials (i.e., $t=0$) and the
obvious expansion of $m_\nu$'s into $M_\beta$'s (i.e., $t=1$). A
key step in our construction is to combine the two expansions described
in the next two items.

\item\emph{$F$-expansion of the $P$ basis.}
One of the main tools for our calculations is the definition of a new
class of tableaux, called \emph{starred tableaux}. With these, we give
in Theorem~\ref{thm:skewP-to-FQ} a combinatorial expansion of the skew
Hall-Littlewood polynomials $P_{\lambda/\mu}(t)$ in terms of the
fundamental quasisymmetric functions $\FQ_\alpha$.  A minor
variation to our method gives a corresponding expansion for the dual
Hall-Littlewood polynomials $Q_{\lambda/\mu}$ (see
Theorem~\ref{thm:skewQ-to-FQ}).

\item\emph{$G$-expansions of the $F$ basis and the $M$ basis.}  In
  Theorems~\ref{thm:MFQG} and~\ref{thm:MMG} we give explicit
  combinatorial expansions for the $\FQ_\alpha$ and the $M_\beta$ in
  terms of the $G_\gamma(t)$.  These are inverse matrices to those
  found in~\cite{hivert}.  See Remark~\ref{rem:nsym} below for the
  relationship to existing results in the realm of noncommutative
  symmetric functions.

\item\emph{$G$-expansions of the $\mathcal{S}$ basis and
    $K_{\alpha}$.}  In Theorems~\ref{thm:MSG} and~\ref{thm:MKG} we
  give explicit combinatorial expansions for the quasisymmetric Schur
  functions $\mathcal{S}_{\beta}$ and the peak quasisymmetric
  functions $K_{\alpha}$ in terms of the $G_\gamma(t)$'s.
\end{enumerate}

\begin{remark}\label{rem:nsym}
  The algebra $\NSym$ of noncommutative symmetric functions, developed
  by Gelfand, Krob, Lascoux, Leclerc, Retakh, and
  Thibon~\cite{gkllrt}, is a dual Hopf algebra to $\QSym$. Many bases
  of $\NSym$ have been developed (see, for
  example,~\cite{hlmvw,LNT,tevlin}).  In papers such as
  ~\cite{bbssz,bz,ntw,tevlin11}, attempts are made to construct bases
  that are suitable analogues in $\NSym$ of the Hall-Littlewood
  polynomials.  Of the transition matrices we describe in
  ~\S\ref{sec:Fexp} and~\S\ref{sec:Gexp}, the only one which we know
  to have been constructed in the dual setting of $\NSym$ is that of
  $\Mtrans(F,G)$: Hivert introduces \emph{noncommutative
    Hall-Littlewood polynomials} $H_\gamma(t)$ that are dual to the
  $G_\gamma(t)$.  If we let $R$ denote the basis for $\NSym$ of
  \emph{ribbon Schur functions}, then it follows that $\Mtrans(F,G) =
  \Mtrans(H,R)^T$ (see~\cite[Theorem 6.13]{ntw}).  Our proof of
  Theorem~\ref{thm:mfg} is substantially the same as Hivert's proof.
  We include our proof both for completeness and so as to give a
  derivation that does not invoke $\NSym$.
\end{remark}

The structure of the paper is as follows.  For ease of reference, we
define all bases discussed here in~\S\ref{sec:review-bases} and
summarize known combinatorial transition matrices
in~\S\ref{sec:rev-mat}.  The expansions of the various polynomials in
terms of the $\FQ_\alpha$'s and the $G_\gamma(t)$'s are presented in
\S\ref{sec:Fexp} and \S\ref{sec:Gexp}, respectively.  
\S\ref{sec:Qpleth} contains a few remarks on transition matrices
for plethystically transformed Hall-Littlewood polynomials.
Finally, specific examples of transition matrices discussed here are
listed in Appendix~\ref{sec:app}.  SAGE code for computing these
transition matrices is available on the third author's
website~\cite{sage-code}.

\section{Review of Symmetric and Quasisymmetric Bases}
\label{sec:review-bases}

This section reviews the definitions of the symmetric and
quasisymmetric functions appearing in Figure~\ref{fig:prism}.
Logically, the precise definitions of the various bases are not needed
in this paper, as the expansions found in~\S\ref{sec:Fexp}
and~\S\ref{sec:Gexp} are derived from the known transition matrices
of~\S\ref{sec:rev-mat}.  However, the material of this section is
included for completeness.  

It will be necessary to introduce a number of functions in the
variables $x_1,\ldots,x_N$, some of which have an extra parameter $t$.
For brevity, we will suppress much of this notation.  For example,
$G_\gamma(x_1,\ldots,x_N;t)$, $G_\gamma(x;t)$, $G_\gamma(t)$, and
$G_\gamma$ all refer to the Hivert quasisymmetric function indexed by
$\gamma$.

\subsection{Compositions and Partitions}
\label{subsec:comp-par}

Given $n\in\N$, a \emph{composition of $n$} is a sequence
$\alpha=(\alpha_1,\alpha_2,\ldots,\alpha_k)$ of positive integers
(called \emph{parts}) with $\alpha_1+\cdots+\alpha_k=n$.  Define the
\emph{length} $\ell(\alpha)$ to be the number of parts of $\alpha$,
and the \emph{size} $|\alpha|$ to be the sum of its parts. 
For example, the composition $\alpha=(2,4,1)$ has
$\ell(\alpha) = 3$ and $|\alpha|=7$.  We may abbreviate the notation,
writing $\alpha$ as $241$, when no confusion can arise.  Let $\Comp_n$ be the
set of compositions of $n$, and let $\Comp$ be the set of all compositions.
A composition
$\lambda=(\lambda_1,\lambda_2,\ldots,\lambda_k)\in\Comp_n$ is called a
\emph{partition of $n$} iff $\lambda_1\geq\lambda_2\geq\cdots \geq
\lambda_k$. We write $\Par_n$ for the set of partitions of $n$ and
$\Par$ for the set of all partitions.  For any composition or partition $\alpha$, we define
$\alpha_j=0$ for all $j>\ell(\alpha)$.

For $n\in\N^+$, there are $2^{n-1}$ compositions of $n$ and $2^{n-1}$
subsets of $[n-1]=\{1,2,\ldots,n-1\}$.  One can define natural
bijections between these sets of objects as follows.  Given
$\alpha\in\Comp_n$ as above, let $$\sub(\alpha)
=\{\alpha_1,\alpha_1+\alpha_2,\alpha_1+\alpha_2+\alpha_3,\ldots,
\alpha_1+\cdots+\alpha_{k-1}\} \subseteq [n-1].$$ The inverse bijection
sends any subset $T=\{t_1<t_2<\cdots<t_m\}\subseteq [n-1]$ to
$$\comp(T)=(t_1,t_2-t_1,t_3-t_2,\ldots,t_m-t_{m-1},n-t_m)\in\Comp_n.$$

Given $\alpha,\beta\in\Comp_n$, we say $\beta$ is \emph{finer} than
$\alpha$, denoted $\beta\succeq\alpha$, iff
$\sub(\alpha)\subseteq\sub(\beta)$. Informally, $\beta$ is finer than
$\alpha$ if we can chop up some of the parts of $\alpha$ into smaller
pieces (without reordering anything) and obtain $\beta$. For example,
$1111\succeq 121 \succeq 31\succeq 4$.

\subsection{Symmetric Polynomials} 
\label{subsec:symm-bases}

Let $\field$ be a commutative ring, and let $\symmgp_N$ denote the
symmetric group on $N$ letters.  A polynomial $f\in \field[x_1,\ldots,x_N]$
is called \emph{symmetric}
iff $$w(f(x)) = f(x_{w(1)},x_{w(2)},\ldots,x_{w(N)})=f(x_1,x_2,\ldots,x_N)
\mbox{ for all $w\in \symmgp_N$.}$$ Write 
$\Sym_N=\Sym_N(\field)$ for the ring of symmetric
polynomials in $N$ variables with coefficients in $\field$;
we usually omit the ground ring $\field$ from the notation. 
For each $n\geq 0$, let $\Sym_N^n=\Sym_N^n(\field)$ be the subspace of $\Sym_N$
consisting of zero and the homogeneous polynomials of degree $n$. For
$N\geq n$, bases of the vector space $\Sym_N^n$ are naturally indexed
by partitions of $n$.

Given $\lambda\in\Par_n$ of length $k\leq N$, the \emph{monomial
  symmetric polynomial} $m_{\lambda}(x_1,\ldots,x_N)$ is the sum of
all distinct monomials that can be obtained by permuting subscripts in
$x^{\lambda}=x_1^{\lambda_1} x_2^{\lambda_2} \cdots x_k^{\lambda_k}$.  
For $N\geq n$, $\{m_{\lambda}(x_1,\ldots,x_N):\lambda\in\Par_n\}$ is readily 
seen to be a basis of $\Sym_N^n$.

Now suppose $N\geq n$ and that $\nu\in\Par_n$ is a partition with
\emph{distinct} parts.  
The \emph{skew-symmetric polynomial 
indexed by $\nu$ in $N$ variables} is
\begin{equation*}
  a_{\nu}(x_1,\ldots,x_N)
  =\sum_{w\in \symmgp_N} \sgn(w) w(x^{\nu}) = \det\|
  x_i^{\nu_j} \|_{1\leq i,j\leq N}. 
\end{equation*}
In particular, letting $\delta_N=(N-1,N-2,\ldots,2,1)$,
$a_{\delta_N}(x_1,\ldots,x_N) = \prod_{1\leq i<j\leq N} (x_i-x_j)$ is
the Vandermonde determinant.  Given $\lambda\in\Par_n$, the
\emph{Schur symmetric polynomial indexed by $\lambda$ in $N$
  variables} is
\begin{equation*}
  s_{\lambda}(x_1,\ldots,x_N)=
 \frac{a_{\lambda+\delta_N}(x_1,\ldots,x_N)}{a_{\delta_N}(x_1,\ldots,x_N)}.
\end{equation*}

For the rest of the paper, let $t$ be an indeterminate, 
and let $\field$ be any
field containing $\Q(t)$ as a subfield.  Following~\cite[\S III.1,
  pp. 204--7]{Macd}, we define the \emph{Hall-Littlewood symmetric
  polynomials} as follows.  Fix $\lambda\in\Par_n$ and $N\geq n$. 
Define $[0]_t=0$ and for $m\geq 1$, $[m]_t=1+t+t^2+\cdots+t^{m-1}$.
Also set $[m]!_t=\prod_{i=1}^m [i]_t$, and $[0]!_t=1$.  Define
$v_\lambda(t) = \prod_{i\geq 0} [m_i]!_t$ where $m_i$ denotes the
number of occurrences of $i$ as a part of $\lambda$.  Then the
Hall-Littlewood polynomial indexed by $\lambda$ is
\begin{equation}\label{eq:P-def}
  P_{\lambda}(x;t) = 
\frac{1}{v_{\lambda}(t)} \sum_{w\in\symmgp_N} w \left(x^{\lambda}
 \prod_{i<j} \frac{x_i-tx_j}{x_i-x_j}\right).  
\end{equation} 

Setting $t=0$ in $P_{\lambda}$ gives $s_{\lambda}$, whereas setting
$t=1$ in $P_{\lambda}$ gives $m_{\lambda}$.  Thus, the Hall-Littlewood
basis ``interpolates'' between the Schur basis and the monomial basis.

One can define Schur polynomials and Hall-Littlewood polynomials more
concretely by giving combinatorial descriptions of their expansions in
terms of monomial symmetric polynomials.  See~\S\ref{subsec:Msm}
and~\S\ref{subsec:MPm} below.

\subsection{Quasisymmetric Polynomials}
\label{subsec:qsym-bases}

A polynomial $f\in \field[x_1,\ldots,x_N]$ is called \emph{quasisymmetric}
iff for every composition $\alpha=(\alpha_1,\ldots,\alpha_k)$ with at
most $N$ parts and every $1\leq i_1<i_2<\cdots <i_k\leq N$, the
monomials $x_1^{\alpha_1}x_2^{\alpha_2}\cdots x_k^{\alpha_k}$ and
$x_{i_1}^{\alpha_1}x_{i_2}^{\alpha_2}\cdots x_{i_k}^{\alpha_k}$ have
the same coefficient in $f$. Write $\QSym_N=\QSym_N(\field)$ for the ring of
quasisymmetric polynomials in $N$ variables with coefficients
in $\field$. For each $n\geq 0$, let
$\QSym_N^n=\QSym_N^n(\field)$ be the subspace of $\QSym_N$ consisting 
of zero and the
homogeneous polynomials of degree $n$. For $N\geq n$, bases of the
vector space $\QSym_N^n$ are naturally indexed by compositions of
$n$. Symmetric polynomials are quasisymmetric, so $\Sym_N^n$ is a
subspace of $\QSym_N^n$.

For $\alpha\in\Comp_n$ of length $k\leq N$, the
\emph{monomial quasisymmetric polynomial}  $M_{\alpha}(x_1,\ldots,x_N)$
is the sum of all monomials $x_{i_1}^{\alpha_1} x_{i_2}^{\alpha_2} \cdots
 x_{i_k}^{\alpha_k}$ for which $1\leq i_1<i_2<\cdots <i_k\leq N$. For $N\geq n$,
$\{M_{\alpha}(x_1,\ldots,x_N):\alpha\in\Comp_n\}$ is readily seen to be
a basis of $\QSym^n_N$.

Next, for $\alpha\in\Comp_n$ with length at most $N$,
define Gessel's \emph{fundamental quasisymmetric polynomial}~\cite{gessel} by
\begin{equation}\label{eq:FQ-def}
\FQ_{\alpha}(x_1,\ldots,x_N)=\sum x_{w_1}x_{w_2}\cdots x_{w_n},
\end{equation} 
where we sum over all subscript sequences $w=w_1w_2\cdots w_n$ such
that $1\leq w_1\leq w_2\leq\cdots\leq w_n\leq N$ and for all
$j\in\sub(\alpha)$, $w_j<w_{j+1}$. In other words, strict increases in
the subscripts are required in the ``breaks'' between parts of the
composition $\alpha$. Call sequences $w$ satisfying these conditions
\emph{$\sub(\alpha)$-compatible}, and write $x^w=x_{w_1}\cdots
x_{w_n}$.  A routine inclusion-exclusion argument
(cf.~\S\ref{subsec:MFQM-MMFQ} below) shows that for $N\geq n$,
$\{\FQ_{\alpha}(x_1,\ldots,x_N):\alpha\in\Comp_n\}$ is a basis of
$\QSym^n_N$.  Note that some authors index fundamental quasisymmetric
polynomials by pairs $n,T$ where $T\subseteq [n-1]$.  Additionally,
various letters ($F$, $L$, $Q$, etc.) have been used to denote these
polynomials.

As in the symmetric case, we would like to have quasisymmetric
Hall-Littlewood polynomials (depending on a parameter $t$) that
interpolate between $\FQ_{\alpha}$ (when $t=0$) and $M_{\alpha}$ (when
$t=1$).  We sketch the definition of one such family of polynomials,
introduced and studied by Hivert~\cite{hivert}.  Quasisymmetric
functions arise as the invariants of a certain action of $\symmgp_N$
on polynomials.  From this action, one can define divided difference
operators in a degenerate Hecke algebra $H_N(0)$ which can then be
lifted to $H_N(q)$.  Hivert's quasisymmetric Hall-Littlewood
polynomials thereby arise from a corresponding $t$-analogue
$\boxdot_\omega$ of the Weyl symmetrizer. For a composition $\alpha$
of length $k\leq N$, define
\begin{equation*}
  G_{\alpha}(x_1,\ldots,x_N;t) = \frac{1}{[k]!_t[N-k]!_t}
    \boxdot_\omega(x_1^{\alpha_1}\cdots x_k^{\alpha_k}).
\end{equation*}
As in the case of symmetric Hall-Littlewood polynomials, there is a
more concrete combinatorial definition of $G_{\alpha}$ giving its
expansion into monomials. We discuss this definition
in~\S\ref{subsec:MGM}.

Hivert's $G_\alpha(t)$ are quasisymmetric versions of the
Hall-Littlewood polynomials $P_\lambda(t)$; when $t=0$, the latter
specialize to Schur polynomials.  In light of these relationships, the
$G_\alpha(t)$ can be viewed as a quasisymmetric $t$-analogue of Schur
functions.  However, a more direct ``quasisymmetric Schur function''
has been introduced by Haglund, Luoto, Mason, and van Willigenburg
\cite{hlmvw} via specializations of nonsymmetric Macdonald polynomials
to Demazure atoms.  These have a combinatorial expansion, which we now
describe.

Given a composition $\alpha$, one forms its \emph{diagram} by placing
$\alpha_i$ boxes, or cells, in the $i$-th row from top to bottom, and
left-justifying the rows. The entries in the diagram of $\alpha$ are
given matrix coordinates $(i,j)$. A filling $T$ of the diagram of
$\alpha$ is a \emph{semistandard composition tableau} ($\SSCT$) if the
following three conditions hold.
\begin{enumerate}
\item[(C1)] The entries in each row are weakly decreasing when read
  from left to right.
\item[(C2)] Entries in the leftmost column of $T$ are strictly
  increasing when read from top to bottom. 
\item[(C3)] Entries satisfy the \emph{triple rule}, namely, if $(i,k)$
  and $(j,k)$ are two cells in the same column, with $i<j$, then:
	\begin{itemize}
	\item if $\alpha_i \ge \alpha_j$, then either $T(j,k)<T(i,k)$
          or $T(i,k-1)<T(j,k)$; 
	\item if $\alpha_i < \alpha_j$, then either $T(j,k)<T(i,k)$ or
          $T(i,k)<T(j,k+1)$.
	\end{itemize}
\end{enumerate}
It follows from these conditions that all entries in the same column
of an SSCT must be distinct.

The \emph{content} of an $\SSCT$ $T$ is 
$\cont(T)=\alpha=(\alpha_1, \alpha_2, \ldots)$,
where $\alpha_i$ is the number of times $i$ appears in $T$. The
corresponding monomial is $x^T = x^{\alpha}$. For
example, the following picture shows an $\SSCT$ of shape $5264$,
content $32313212$, and monomial $x_1^3 x_2^2 x_3^3 x_4 x_5^3 x_6^2
x_7 x_8^2$.
\begin{equation*}
  \young(32211,41,755533,8866)
\end{equation*}

A \emph{quasisymmetric Schur function} $\mathcal{S}_{\alpha}$ is
defined as the generating series of semistandard composition tableaux,
i.e., $\mathcal{S}_{\alpha} = \sum x^T$, where the sum
runs over all $\SSCT$ of shape~$\alpha$.  Part of their importance stems
from the fact that they provide refinements of Schur functions,
as in the formula
\[ 
s_\lambda = \sum_{\sort(\alpha) = \lambda} \mathcal{S}_\alpha,
\]
where $\sort(\alpha)$ is the partition obtained by organizing the
parts of $\alpha$ from largest to smallest.

As discussed, the Hall-Littlewood functions specialize to Schur
functions and monomial symmetric functions when $t$ equals $0$ and
$1$, respectively.  The \emph{Schur $P$-functions} are the $t=-1$
specializations. They are indexed by strict partitions, namely,
partitions into distinct parts.  The Schur $P$-functions are another
example of a family of symmetric functions with combinatorial
expansions in terms of both a collection of semistandard objects and a
collection of standard objects. The first one is the classical
expansion of the Schur $P$-functions as a generating series for
\emph{shifted semistandard Young tableaux}, where each tableau is
weighted by its corresponding monomial (see, e.g.,
\cite{Hai,Sag,Ste1}). The second one is the expansion given by
Stembridge \cite{Ste2} of the Schur $P$-functions as a sum over
\emph{shifted standard Young tableaux} of the corresponding \emph{peak
  quasisymmetric functions}, denoted $K_{\alpha}$.  The $K_\alpha$
span an important subalgebra of $\QSym$ called the \emph{peak
  quasisymmetric algebra}.  While we will not be concerned with Schur
$P$-functions in this paper, transition matrices involving the peak
quasisymmetric functions will be presented.

Let $\Comp'_n$ be the set of compositions of $n$ with no parts of
length $1$, except perhaps the last one.  Note that $\alpha \in \Comp'_n$ 
if and only if $\sub(\alpha)$ is a subset of $[n-1]$ with
no consecutive elements.  For each $\alpha\in \Comp'_n$,
Stembridge~\cite{Ste2} defines the peak quasisymmetric function
$K_\alpha$ as the generating series of certain \emph{enriched
  $P$-partitions}.  We direct the reader to \cite{Ste2} for this
definition and motivation.  In this paper, however, we will use as a
definition their expansion into fundamental quasisymmetric functions.  

For $B \subseteq [n-1]$, let $B+1 = \{b+1:\, b\in B\}\setminus \{n\}$.
Furthermore, we write $\bigtriangleup$ for the symmetric difference
between two sets, namely, $A \bigtriangleup B = (A \cup B) \setminus
(A \cap B)$.  Then, as in \cite[Proposition 3.5]{Ste2}
for $\alpha \in \Comp'_n$, let
\begin{equation*}
  K_\alpha = 
  \sum_{\beta:\ \sub(\alpha) \subseteq \sub(\beta) \bigtriangleup (\sub(\beta)+1)} \FQ_\beta.
\end{equation*}

\section{Review of Known Transition Matrices}
\label{sec:rev-mat}

In the theory of symmetric and quasisymmetric polynomials,
much combinatorial information is encoded in the transition
matrices between various bases. Given two bases 
$B=\{B_{\lambda}: \lambda\in\Par_n\}$ and
$C=\{C_{\lambda}: \lambda\in\Par_n\}$ of $\Sym_N^n$,
the \emph{transition matrix} $\Mtrans(B,C)$ is the unique matrix 
(with entries in $\field$ and rows and columns indexed by partitions of $n$) 
such that
\begin{equation*}
  B_{\lambda}=\sum_{\mu\in\Par_n} \Mtrans(B,C)_{\lambda,\mu}C_{\mu}.
\end{equation*}
Given a third basis $D$, 
it follows readily that $\Mtrans(B,D)=\Mtrans(B,C)\Mtrans(C,D)$
and $\Mtrans(C,B)=\Mtrans(B,C)^{-1}$. We define $\Mtrans(B,C)$ similarly if $B$ and $C$
are bases of $\QSym_N^n$, but here the rows and columns of the
matrix are indexed by compositions of $n$. Finally, if $B$ is a basis
of $\Sym_N^n$ and $C$ is a basis of $\QSym_N^n$, then $\Mtrans(B,C)$
is a rectangular matrix expressing each $B_{\lambda}$ as an $\field$-linear
combination of the $C_{\alpha}$'s.

This section gives combinatorial formulas for previously known
transition matrices associated to some of the edges in
Figure~\ref{fig:prism}. Transitions to the monomial bases offer
alternate explicit definitions for Schur polynomials and the various
forms of Hall-Littlewood polynomials. 
Specific examples of these transition matrices appear in
Appendix~\ref{sec:app}.  SAGE code for computing these transition
matrices is available on the third author's website~\cite{sage-code}.

\subsection{$\boldsymbol{\Mtrans(s,m)}$}
\label{subsec:Msm}

The expansion of Schur polynomials into monomials uses
semistandard tableaux. For later work, we will also need tableaux
of skew shape. Suppose $\lambda,\nu\in\Par$ satisfy
$\nu\subseteq\lambda$, i.e., $\nu_i\leq \lambda_i$ for all $i$. 
Define the \emph{skew diagram}
\[ \lambda/\nu=\{(i,j)\in\N^+\times\N^+:
  1\leq i\leq\ell(\lambda), \nu_i<j\leq \lambda_i \}. \]
We will draw skew diagrams using the English convention where
 the longest rows are at the top. For $N\in\N^+$, 
a \emph{semistandard tableau} of shape $\lambda/\nu$ with
entries in $[N]=\{1,2,\ldots,N\}$ is a function 
$T:\lambda/\nu\rightarrow [N]$ that is weakly increasing along rows
and strictly increasing down columns. Writing $n=|\lambda/\nu|$,
a \emph{standard tableau} of shape $\lambda/\nu$ is a bijection
$S:\lambda/\nu\rightarrow [n]$ that is also a semistandard tableau.
Let $\ssyt_N(\lambda/\nu)$ be the set of all semistandard tableaux
of shape $\lambda/\nu$ with entries in $[N]$, 
and let $\syt(\lambda/\nu)$ be the set of all
standard tableaux of shape $\lambda/\nu$.  
For any $T\in\ssyt_N(\lambda/\nu)$, 
the \emph{content of $T$} is the composition
$\cont(T)=(\alpha_1,\ldots,\alpha_N)$, where
$\alpha_i$ is the number of times $i$ appears in a cell of $T$.  
The \emph{content monomial} $x^{T}$
is $x_1^{\alpha_1}x_2^{\alpha_2}\cdots x_N^{\alpha_N}$.

The \emph{skew Schur polynomial} in $N$ variables can now be defined as
\[ s_{\lambda/\nu}(x_1,\ldots,x_N)
 =\sum_{T\in\ssyt_N(\lambda/\nu)} x^{T}. \]
The ordinary Schur polynomial $s_{\lambda}$ is obtained by taking
$\nu=(0)$ here. Skew Schur polynomials are symmetric, although this is
not obvious from the combinatorial definition. Consequently, we 
have the following expansion of Schur polynomials in terms of
the monomial symmetric polynomials.

\begin{theorem}\label{thm:Msm}
For all $\lambda,\mu\in\Par_n$, $\Mtrans(m,s)_{\lambda,\mu}$
is the number of semistandard tableaux of shape $\lambda$ and content $\mu$.  
(This number is also called the \emph{Kostka number} $K_{\lambda,\mu}$.)
\end{theorem}

\subsection{$\boldsymbol{\Mtrans(m,s)}$} 
\label{subsec:Mms}

E\u{g}ecio\u{g}lu and Remmel~\cite{ER-kinv} found the following
combinatorial formula for the inverse Kostka matrix $\Mtrans(m,s)$.  Fix
$\mu\in\Par$. A \emph{special rim-hook} is a sequence of cells in the
diagram of $\mu$ that begins in the leftmost column and moves up and
right through the diagram. The sign of a rim-hook occupying $r$ rows
is $(-1)^{r-1}$. A \emph{special rim-hook tableau} $S$ of shape $\mu$
is a dissection of the diagram of $\mu$ into a disjoint union of
special rim-hooks.  The \emph{sign} of $S$, $\sgn(S)$, is the product
of the signs of the rim-hooks in it.  The \emph{type} of $S$ is the
integer partition obtained by listing the lengths of the rim-hooks in
$S$ in decreasing order.

\begin{theorem}\label{thm:Mms}\cite[Theorem 1]{ER-kinv}
For all $\lambda,\mu\in\Par_n$, $\Mtrans(m,s)_{\lambda,\mu}=\sum_S \sgn(S)$ 
summed over all special rim-hook tableaux $S$ of shape $\mu$ and type $\lambda$.
\end{theorem}

\subsection{$\boldsymbol{\Mtrans(s,P)}$} 
\label{subsec:MsP}

Lascoux and Sch\"utzenberger~\cite{LS-chg} first discovered a
combinatorial formula for the ``$t$-Kostka matrix'' $\Mtrans(s,P)$ involving
the famous ``charge'' statistic; many details were subsequently
supplied by Butler~\cite{butler}.  Given a permutation $w$ of $[n]$, let
$\IDes(w)$ be the set of $k<n$ such that $k+1$ appears to the left of
$k$ in $w$, and let $\chg(w)=\sum_{k\in\IDes(w)} (n-k)$.

Next, let $v$ be a word of partition content (i.e., for all $k\geq 1$,
the number of $(k+1)$'s in $v$ is no greater than the number of $k$'s).
Extract one or more permutations from $v$ as follows. Scan $v$ from
left to right marking the first $1$, then the first $2$ after that, etc., 
returning to the beginning of $v$ when the right end is reached. Do this 
until the largest symbol has been marked. Remove the marked symbols
from $v$ (in the order they appear) to get the first permutation. Continue
to extract permutations in this way until all symbols of $v$ have been used,
and let $\chg(v)$ be the sum of the charges of the associated permutations.  
Finally, given a semistandard tableau $T$ of partition content,
let $w(T)$ be the word obtained by reading symbols row by row from
top to bottom (i.e., longest row first), reading each row from right to left.
Then define $\chg(T)=\chg(w(T))$.

\begin{theorem}\cite{butler,LS-chg}
For all $\lambda,\mu\in\Par_n$, $\Mtrans(s,P)_{\lambda,\mu}=\sum_T t^{\chg(T)}$
summed over all semistandard tableaux $T$ of shape $\lambda$ and content $\mu$.
\end{theorem}

\subsection{$\boldsymbol{\Mtrans(P,s)}$} 
\label{subsec:MPs}

Carbonara~\cite{carbonara} gave a combinatorial reformulation
of~\eqref{eq:P-def} that describes entries of the ``inverse $t$-Kostka
matrix'' $\Mtrans(P,s)$ in terms of special tournament matrices.  An
$n\times n$ \emph{tournament matrix} is a matrix $B$ with entries in
$\{0,1\}$ such that $B_{i,i}=0$ for all $i$ and, for all $i\neq j$,
exactly one of $B_{i,j}$ and $B_{j,i}$ equals $1$. Given
$\lambda,\mu\in\Par$ having length at most $n$, we say the tournament
matrix $B$ has \emph{type $\lambda$ and shape $\mu$} iff the sequence
$(\lambda_i+\sum_{j=1}^n B_{i,j}:1\leq i\leq n)$ is a rearrangement of
the sequence $(\mu_i+n-i:1\leq i\leq n)$. Such a matrix is
\emph{special} (for $\lambda$) iff for all $i<j$ with
$\lambda_i=\lambda_j$, $B_{ij}=1$.  Let $\Gamma^*_{\lambda,\mu}$ be
the set of all $n\times n$ special tournament matrices of type
$\lambda$ and shape $\mu$, where $n=\max(\ell(\lambda),\ell(\mu))$.

We define signs and weights for $B\in\Gamma^*_{\lambda,\mu}$ as follows.
Since the entries of $(\mu_i+n-i:1\leq i\leq n)$ are distinct,
there is a unique $w\in \symmgp_n$ such that 
$\lambda_i+\sum_{j=1}^n B_{i,j}=\mu_{w(i)}+n-w(i)$ for $1\leq i\leq n$.
Define $\sgn(B)=\sgn(w)$, where $\sgn(w)$ is the usual sign of 
the permutation $w$. Define $\wt(B)=\sum_{i>j} B_{i,j}$, which is
the number of nonzero entries of $B$ below the diagonal.

\begin{theorem}~\cite[Theorem 2]{carbonara}
For all $\lambda,\mu\in\Par_n$, 
$\Mtrans(P,s)_{\lambda,\mu}=\sum_{B\in\Gamma^*_{\lambda,\mu}} \sgn(B)(-t)^{\wt(B)}$.
\end{theorem}

\subsection{$\boldsymbol{\Mtrans(P,m)}$}
\label{subsec:MPm} 

Macdonald~\cite[\S III.5, p. 229]{Macd} gives a formula
for the monomial expansion of \emph{skew Hall-Littlewood polynomials}
$P_{\lambda/\nu}(x_1,\ldots,x_N;t)$ which yields $\Mtrans(P,m)$ 
by taking $\nu=(0)$. We introduce the following combinatorial
model for Macdonald's formula.

Assume $\lambda/\nu$ is a skew shape with $N\geq\ell(\lambda)$.  Given
$T\in\ssyt_N(\lambda/\nu)$, define a cell $c$ with entry $v=T(c)$ to
be \emph{special} for $T$ iff $c$ is not in column 1 and there are no
$v$'s in the column of $T$ just left of $c$'s column. In this case,
the \emph{weight} of $c$ is the number of cells weakly below $c$ in
the column just left of $c$ that either have entries less than $v$ or
are part of the diagram for $\nu$.  Formally, the set of
\emph{special cells for $T$} is
\begin{equation*}
\spec(T)=\{ (i,j)\in\lambda/\nu: j>1\mbox{ and for all $u$
 with $(u,j-1)\in\lambda/\nu$, } T(u,j-1)\neq T(i,j) \}.
\end{equation*}
The \emph{weight} of a special cell $(i,j)$ is
\begin{multline*}
  \wt(i,j) = |\{(u,j-1)\in\lambda/\nu: u \geq i\mbox{ and
  } T(u,j-1)<T(i,j)\}|\\ + |\{(u,j-1)\in\nu/(0): u\geq i\}|.
\end{multline*}

Now define the set of \emph{starred semistandard tableaux}
\begin{equation*}
\ssyt_N^*(\lambda/\nu)=
\{(T,E):T\in\ssyt_N(\lambda/\nu)\mbox{ and }E\subseteq\spec(T)\}.
\end{equation*} 
A starred tableau $T^*=(T,E)$ has \emph{sign} $\sgn(T^*)=(-1)^{|E|}$,
\emph{$t$-weight} $\tstat(T^*)=\sum_{c\in E} \wt(c)$,  
\emph{content} $\cont(T^*)=\cont(T)$, \emph{$x$-weight} $x^{T^*}=x^T$,
and \emph{overall weight} $\sgn(T^*)t^{\tstat(T^*)}x^{T^*}.$ 

For $T\in\ssyt_N(\lambda/\nu)$, Macdonald defines
$\psi_T(t)=\prod_{c\in\spec(T)} (1-t^{\wt(c)})$. 
Then Macdonald's monomial expansion of the skew Hall-Littlewood
polynomials is
\begin{equation*} 
  P_{\lambda/\nu}(x_1,\ldots,x_N;t)
  =\sum_{T\in\ssyt_N(\lambda/\nu)} \psi_T(t)x^T.
\end{equation*}
Expanding the product in $\psi_T(t)$ using the distributive law, we
get $\sum_{E\subseteq\spec(T)}\prod_{c\in E} (-t^{\wt(c)})$. Comparing
to the overall weight of starred tableaux, we find that
\begin{equation}\label{eq:ssyt-sum}
 P_{\lambda/\nu}(x_1,\ldots,x_N;t)=\sum_{T^*\in \ssyt_N^*(\lambda/\nu)} 
   \sgn(T^*)t^{\tstat(T^*)}x^{T^*}. 
\end{equation}
Since $P_{\lambda/\nu}$ is known to be a symmetric polynomial,
we deduce the following formula.

\begin{theorem}
For all $\lambda,\mu\in\Par_n$, 
$\Mtrans(P,m)_{\lambda,\mu}=\sum_{T^*} \sgn(T^*)t^{\tstat(T^*)}$
summed over all starred semistandard tableaux $T^*$ of shape $\lambda$
and content $\mu$.  
\end{theorem}

\begin{example}\label{ex:sstd-elev} 
  Let $\lambda=8654$, $\nu=0$, $N\geq 8$, and
\begin{equation}\label{eq:Tex}
T=\tableau{
1&1&1&\underline{2}&2&\underline{4}&\underline{5}&5 \\
2&2&3&3&\underline{6}&\underline{8} \\
3&3&\underline{4}&4&\underline{7} \\
5&5&5&5}.
\end{equation}

In~\eqref{eq:Tex}, the special cells are indicated by the underlined entries.
Specifically, 
\begin{equation*}
  \spec(T)=\{(1,4),(1,6),(1,7),(2,5),(2,6),(3,3),(3,5)\}.
\end{equation*} 
These special cells have respective weights $1,1,1,3,2,1,2$. So $T$
contributes the term $(1-t)^4(1-t^2)^2(1-t^3)x^T$ to $P_{\lambda}$.
A typical starred tableau is $T^*=(T,\{(1,4),(1,6),(2,6)\})$.  It 
can be pictured as follows:
\begin{equation*}
T^*=\tableau{
1&1&1&2^*&2&4^*&5&5 \\
2&2&3&3&6&8^* \\
3&3&4&4&7 \\
5&5&5&5}.
\end{equation*} 
The overall weight of this object is
$(-1)^3t^{1+1+2}x_1^3x_2^4x_3^4x_4^3x_5^6x_6x_7x_8=-t^4x^T$.
\end{example} 

\subsection{$\boldsymbol{\Mtrans(s,\FQ)}$} 
\label{subsec:MsFQ}

The fundamental quasisymmetric expansion of Schur polynomials is a sum
over standard tableaux, rather than semistandard tableaux.  Given
$\lambda\in\Par_n$ and $S\in\syt(\lambda)$, define the \emph{descent
  set} $\Des(S)$ to be the set of $k<n$ such that $k+1$ appears in a
lower row of $S$ than $k$. Define the descent composition
$\Des'(S)=\comp(\Des(S))$
to be the composition associated to this subset of $[n-1]$. 
Gessel~\cite{gessel} first proved that 
\begin{equation*} 
 s_{\lambda}(x_1,\ldots,x_N)=\sum_{S\in\syt(\lambda)} 
   \FQ_{\Des'(S)}(x_1,\ldots,x_N). 
\end{equation*}
This formula can be proved bijectively by identifying the individual
monomials in $\FQ_{\Des'(S)}$ as the content monomials of semistandard
tableaux of shape $\lambda$ that ``standardize'' to $S$.  In terms
of transition matrices, we can state Gessel's result as follows.

\begin{theorem}
For all $\lambda\in\Par_n$ and $\alpha\in\Comp_n$, 
$\Mtrans(s,\FQ)_{\lambda,\alpha}$ is the number of standard tableaux with
shape $\lambda$ and descent set $\sub(\alpha)$.  
\end{theorem}

\subsection{$\boldsymbol{\Mtrans(m,M)}$}
\label{subsec:MmM}

For $\lambda\in\Par$, it is immediate that
$m_{\lambda}=\sum_{\alpha} M_{\alpha}$ summed over all compositions $\alpha$
whose parts can be sorted to give the parts of $\lambda$. Therefore:

\begin{theorem}
For all $\lambda\in\Par_n$ and $\alpha\in\Comp_n$, 
$\Mtrans(m,M)_{\lambda,\alpha}$ is 
$1$ if $\sort(\alpha)=\lambda$ and $0$ otherwise.  
\end{theorem}

\subsection{$\boldsymbol{\Mtrans(\FQ,M)}$ and $\boldsymbol{\Mtrans(M,\FQ)}$}
\label{subsec:MFQM-MMFQ}

\begin{theorem}
For all $\alpha,\beta\in\Comp_n$,
$\Mtrans(\FQ,M)_{\alpha,\beta}$ is $1$ if $\beta\succeq\alpha$ and $0$
otherwise; whereas
$\Mtrans(M,\FQ)_{\alpha,\beta}$ is $(-1)^{\ell(\beta)-\ell(\alpha)}$
if $\beta\succeq\alpha$ and $0$ otherwise.  
\end{theorem} 
\begin{proof} 
(sketch) Using~\eqref{eq:FQ-def}, one may verify that
$\FQ_{\alpha}=\sum_{\beta\succeq\alpha} M_{\beta}$,
which is equivalent to the stated formula for $\Mtrans(\FQ,M)_{\alpha,\beta}$.  
One obtains the formula for $\Mtrans(M,\FQ)=\Mtrans(\FQ,M)^{-1}$ using
the M\"obius inversion theorem on the poset of subsets of $[n-1]$ 
ordered by set inclusion.
\end{proof}

\subsection{$\boldsymbol{\Mtrans(G,\FQ)}$} 
\label{subsec:MGFQ}

Let $\alpha,\beta\in\Comp_n$ with $\beta$ finer than $\alpha$.  Say
$\ell(\alpha)=k$ and $\ell(\beta)=m$. By definition, there exist
indices $0 = i_0 < i_1 < \cdots < i_k = m$ such that $\alpha_j =
\beta_{i_{j-1}+1} + \cdots + \beta_{i_j}$ for $1\leq j\leq k$.  The
\emph{refining composition} $\Bre(\beta,\alpha) =
(i_1-i_0,i_2-i_1,\ldots,i_k-i_{k-1})$ records the number of parts of
$\beta$ derived from each part of $\alpha$. Define $s(\alpha,\beta) =
\sum_{j=1}^k j(i_j-i_{j-1}-1)$.  Note that in the notation
$\Bre(\beta,\alpha)$ from~\cite{hivert}, the finer composition is
listed \emph{first}, but in the function $s$ (and $g$, $\xi$ defined
in \S\ref{subsec:MFQG}), we list the finer composition \emph{second}.
This ordering is more convenient when working with transition
matrices. In~\cite[Theorem 6.6]{hivert}, Hivert showed that
for all $N\geq n$ and $\alpha\in\Comp_n$, 
  \begin{equation*}
  G_\alpha(x_1,\ldots,x_N;t) = \sum_{\beta:\beta \succeq \alpha}
  (-1)^{\ell(\beta)-\ell(\alpha)} t^{s(\alpha,\beta)} 
    \FQ_\beta(x_1,\ldots,x_N).
  \end{equation*}
In other words:

\begin{theorem}\label{thm:MGFQ}
For all $\alpha,\beta\in\Comp_n$,
$\Mtrans(G,\FQ)_{\alpha,\beta}$
is $(-1)^{\ell(\beta)-\ell(\alpha)} t^{s(\alpha,\beta)}$
if $\beta\succeq\alpha$ and $0$ otherwise.
\end{theorem}

\begin{example}
Take $\beta = 1221431211$ and $\alpha = 55314$.  Then
$\Bre(\beta,\alpha) = 32113$ and \[s(\alpha,\beta) = 1\cdot 2 + 2\cdot
1 + 3\cdot 0 + 4\cdot 0 + 5\cdot 2 = 14.\] So $\Mtrans(G,\FQ)_{\alpha,\beta} =
(-1)^5t^{14}$.
\end{example}

\subsection{$\boldsymbol{\Mtrans(G,M)}$}  
\label{subsec:MGM}

In~\cite[eq. (105)]{hivert}, Hivert showed that
for all $\alpha\in\Comp_n$ and $N\geq n$,
\[ G_{\alpha}(x_1,\ldots,x_N;t)
  =\sum_{\beta:\beta\succeq\alpha} M_{\beta}(x_1,\ldots,x_N;t)
    \prod_{i=1}^{\ell(\Bre(\beta,\alpha))} (1-t^i)^{\Bre(\beta,\alpha)_i-1}. \] 
Hence:
\begin{theorem} 
For all $\alpha,\beta\in\Comp_n$,
$\Mtrans(G,M)_{\alpha,\beta}$ is
$\prod_{i=1}^{\ell(\Bre(\beta,\alpha))} (1-t^i)^{\Bre(\beta,\alpha)_i-1}$ 
if $\beta\succeq\alpha$ and $0$ otherwise.  
\end{theorem}

\subsection{$\boldsymbol{\Mtrans(\mathcal{S},M)}$}

As defined in~\cite{hlmvw}, $\mathcal{S}_\alpha=\sum_T x^T$ 
summed over all $\SSCT$ of shape $\alpha$. So:
\begin{theorem} 
Let $\SSCT(\alpha,\beta)$ denote the set of all semistandard composition
tableaux of shape $\alpha$ and content $\beta$.
For all $\alpha,\beta\in\Comp_n$,
$\Mtrans(\mathcal{S},M)_{\alpha,\beta}=|\SSCT(\alpha,\beta)|$.  
\end{theorem}

\subsection{$\boldsymbol{\Mtrans(\mathcal{S},F)}$}

A \emph{standard composition tableau ($\SCT$)} $T$ is a $\SSCT$ with content
$(1,1,\ldots,1)$. The \emph{descent set} $\Des(T)$ is the set of those
$i$ that lie weakly to the left of $i+1$. The \emph{descent composition}
for $T$ is $\comp(\Des(T))$.  
For example, given
\begin{equation*}
  \Yvcentermath1 T=\young(85421,93,{\fifteen}{\twelve}{\eleven}{\ten}76,{\seventeen}{\sixteen}{\fourteen}{\thirteen}),
\end{equation*}
$\Des(T) = \{3,5,8,9,12,15\}$ and $\comp(\Des(T))=(3,2,3,1,3,3,2)$.

\begin{theorem}\cite[Theorem 6.2]{hlmvw}\label{thm:msf}
Let $\SCT(\alpha,\beta)$ denote the set of all standard composition
tableaux of shape $\alpha$ and descent composition $\beta$.
For all $\alpha,\beta\in\Comp_n$,
$\Mtrans(\mathcal{S},F)_{\alpha,\beta} = |\SCT(\alpha,\beta)|.$ 
\end{theorem}

\begin{example}
  Consider the $\FQ$-expansion of $\mathcal{S}_{13}$.  The only two
  $\SCT$ of shape $13$ are
  \begin{equation*}
    \Yvcentermath1 \young(1,432) \qquad \textrm{and} \qquad \young(2,431),
  \end{equation*}
  with descent compositions $13$ and $22$. So
  $\mathcal{S}_{13} = \FQ_{13} + \FQ_{22}$.
\end{example}

\subsection{$\boldsymbol{\Mtrans(K,F)}$ and $\boldsymbol{\Mtrans(K,M)}$}

Stembridge \cite{Ste2} expressed the peak quasisymmetric functions as
sums of both fundamental and monomial quasisymmetric functions.
\begin{theorem}\cite[Proposition 3.5]{Ste2}
For all $\alpha\in\Comp'_n$ and $\beta\in\Comp_n$,
  \begin{align}
    \Mtrans(K,F)_{\alpha,\beta} &= 
    \begin{cases} 1, &\textrm{ if } 
      \sub(\alpha) \subseteq \sub(\beta) \bigtriangleup (\sub(\beta)+1) ;\\ 
      0, & \text{ otherwise. }
    \end{cases}\\
    \Mtrans(K,M)_{\alpha,\beta} &= 
    \begin{cases} 2^{\ell(\beta)-\ell(\alpha)}, & 
      \textrm{ if } \sub(\alpha) \subseteq \sub(\beta) \cup (\sub(\beta)+1) ;\\
      0, & \text{ otherwise.}
    \end{cases}
  \end{align}
\end{theorem}
\begin{example}
\begin{equation*}
  K_{31}=\FQ_{31}+\FQ_{22}+\FQ_{121}+\FQ_{112} 
= M_{31}+M_{22}+2M_{211}+2M_{121}+2M_{112}+4M_{1111}.  
\end{equation*}
\end{example}

\section{$\FQ$-expansion of Skew Hall-Littlewood Polynomials}
\label{sec:Fexp}

\subsection{Expansion of $P_{\lambda/\mu}$}
\label{subsec:MPFQ-skew}

Recall from~\S\ref{subsec:MPm} the combinatorial formula~\eqref{eq:ssyt-sum}
for the monomial expansion of the skew Hall-Littlewood polynomials
$P_{\lambda/\mu}(x_1,\ldots,x_N;t)$.  
This section converts this formula to an expansion of these polynomials
in terms of the fundamental quasisymmetric basis. In particular, this
provides a combinatorial interpretation for the entries of $\Mtrans(P,\FQ)$.
We remark that one can also obtain $\Mtrans(P,\FQ)$ by
multiplying the known matrices $\Mtrans(P,s)$ and $\Mtrans(s,\FQ)$.  However, 
this produces a quite complicated interpretation for the coefficients
in $\Mtrans(P,\FQ)$ as signed combinations of standard tableaux and
special tournaments. The new interpretation developed below is much simpler.

To state our result, we need a few more definitions.  Given a skew
diagram $\lambda/\mu$ with $n$ cells, let $\syt^*(\lambda/\mu)$ be the
set of starred tableaux $S^*=(S,E)$ such that $S$ is a
standard tableau of shape $\lambda/\mu$.  In this case, observe that
$\spec(S)$ consists of all cells in the diagram not in column 1.  So
$E$ can be an arbitrary subset of cells of $\lambda/\mu$ not in column
1. Define the \emph{ascent set} of $S^*$, denoted $\Asc(S^*)$, to be
the set of all $k<n$ such that either (a) $k+1$ appears in $S$ in a
lower row than $k$, or (b) there exist $u,i,j$ with $S(u,j-1)=k$,
$S(i,j)=k+1$, and $(i,j)\in E$.  The second alternative says that
$k+1$ appears in a cell of $E$ located in the next column after the
column containing $k$.  Define $\Asc'(S^*)=\comp(\Asc(S^*))$ to be the
associated composition.

\begin{theorem}\label{thm:skewP-to-FQ}
For all skew shapes $\lambda/\mu$, 
\begin{equation*} 
 P_{\lambda/\mu}(x_1,\ldots,x_N;t)
  =\sum_{S^*\in\syt^*(\lambda/\mu)} \sgn(S^*)t^{\tstat(S^*)}
     \FQ_{\Asc'(S^*)}(x_1,\ldots,x_N).
\end{equation*}
In particular, for all $\lambda\in\Par_n$ and $\alpha\in\Comp_n$,
$\Mtrans(P,\FQ)_{\lambda,\alpha}=\sum_{S^*} \sgn(S^*)t^{\tstat(S^*)}$
summed over all starred standard tableaux $S^*$ having shape $\lambda$
and ascent set $\sub(\alpha)$.
\end{theorem}
\begin{proof}
Let $Y$ be the set of pairs $(S^*,w)$ where $S^*=(S,E)\in\syt^*(\lambda/\mu)$
and $w$ is an $\Asc(S^*)$-compatible word of length $n$
(see~\S\ref{subsec:qsym-bases}). The overall weight of $(S^*,w)\in Y$
is $\sgn(S^*)t^{\tstat(S^*)}x^w$.  
Keeping in mind the definition of $\FQ_{\alpha}$, 
we see that the generating function for the weighted set $Y$ is 
\begin{equation}\label{eq:Ysum}
  \sum_{y\in Y} \wt(y)
 =\sum_{S^*\in\syt^*(\lambda/\mu)} 
\sgn(S^*)t^{\tstat(S^*)}\FQ_{\Asc'(S^*)}(x_1,\ldots,x_N).
\end{equation}

Comparing~\eqref{eq:ssyt-sum} and~\eqref{eq:Ysum},
the theorem will be proved if we can construct
a sign-preserving, weight-preserving bijection 
$f:\ssyt^*(\lambda/\mu)\rightarrow Y$.

\begin{example}\label{ex:std-elev} 
  Let $\lambda=65211$, and consider the starred standard tableau
  \begin{equation*}
\cellsize=3ex
S^*=\tableau{
    1&2&3&8^*&9&{15\!}^* \\
    4&6&7^*&{13\!}^*&14 \\
    5&12\\
    10\\
    11}.
\end{equation*}
For this object, $E=\{(1,4),(1,6),(2,3),(2,4)\}$, 
  $\Asc(S^*)=\{3,4,6,7,9,10,14\}$,  and $\Asc'(S^*)=31212141$.
A typical object in $Y$ is $(S^*,w)$ where 
\[ w=111<2<33<4<55<6<7888<9. \] The marked ascents in
  $w$ are mandated for $w$ to be $\Asc(S^*)$-compatible, and there
  is an extra ascent between $w_{11}=7$ and $w_{12}=8$.  The overall
  weight of $(S^*,w)$ is
  $$(-1)^4t^{2+2+1+1}x_1^3x_2x_3^2x_4x_5^2x_6x_7x_8^3x_9=t^6x^w.$$
\end{example}

Continuing the proof, we must define a ``standardization map''
$f:\ssyt^*(\lambda/\mu)\rightarrow Y$ and an ``unstandardization map''
$g:Y\rightarrow \ssyt^*(\lambda/\mu)$ such that $f$ preserves signs and 
weights, $f\circ g=\id_Y$, and $g\circ f=\id_{\ssyt^*(\lambda/\mu)}$. 
Given a semistandard tableau $T\in\ssyt_N(\lambda/\mu)$, recall
that there is a standard tableau $\stdz(T)$, called the \emph{standardization 
of $T$}, defined as follows.
Suppose the entries in $T$ consist of $m_1$ ones, $m_2$ twos, etc.
Define $M_0=0$ and $M_i=m_1+m_2+\cdots+m_i$ for $1\leq i\leq N$.
We obtain $\stdz(T)$ from $T$ by replacing the $m_i$ occurrences of $i$
in $T$, from left to right, by the integers $M_{i-1}+1,M_{i-1}+2,
\ldots,M_i$. Now, given $T^*=(T,E)\in\ssyt^*(\lambda/\mu)$,
define $f(T^*)=(S^*,w)=((\stdz(T),E),w)$, where $w$ consists of the symbols 
in $T$ in increasing order.

We must check that $f(T^*)\in Y$ for all $T^*=(T,E)\in\ssyt^*(\lambda/\mu)$.
First, $S=\stdz(T)$ is a standard tableau of shape $\lambda/\mu$.
Second, since $E\subseteq\spec(T)$, every cell of $E$ is not
in column 1, and so $E\subseteq\spec(S)$. Third, we claim $w$
is an $\Asc(S^*)$-compatible sequence. To check this,
assume $k<n$ and $w_k=w_{k+1}=v$. By definition of standardization,
the unique occurrences of $k$ and $k+1$ in $S$ were used to relabel cells
that both originally contained occurrences of $v$ in $T$. Since the
horizontal strip of cells in $T$ containing $v$ gets relabeled from
left to right, $k+1$ cannot appear in $S$ in a lower row than $k$.
So condition (a) in the definition of $\Asc(S^*)$ does not hold for $k$.
Can condition (b) hold for $k$? If so, there are $u,i,j$ with
$S(u,j-1)=k$ and $S(i,j)=k+1$ and $(i,j)\in E$. But then
$T(u,j-1)=v=T(i,j)$ means that $(i,j)$ is not a special cell
for $T$, which contradicts $E\subseteq\spec(T)$. So we conclude
that $k\not\in\Asc(S^*)$, proving the claim.

With notation as above, observe that $x^{T^*}=x^T=x^w$, so that $f$
preserves the $x$-weight. Since applying $f$ does not change $E$,
$f$ preserves signs. Finally, suppose $(i,j)\in E$ with $T(i,j)=v$.
Since $v$ cannot appear in column $j-1$ of $T$, it follows from
the definition of standardization that
the cells contributing to $\wt(i,j)$ in the computation of
$\tstat(T^*)$ are precisely the cells contributing to $\wt(i,j)$
in the computation of $\tstat(S^*)$. So, $f$ preserves the $t$-weight.

All that remains is to define the two-sided inverse $g$
for $f$. Given $(S^*,w)=((S,E),w)\in Y$, suppose $w$ consists of $m_1$ ones
followed by $m_2$ twos, etc. With $M_i$ defined as before, let 
$T:\lambda/\mu\rightarrow [N]$ be obtained from $S$ by replacing
all symbols $M_{i-1}+1,\ldots,M_i$ by $i$'s, for $1\leq i\leq N$.
Then set $g(S^*,w)=(T,E)$. We must check that $(T,E)$ lies in the
codomain $X$. Note that $w$ is $\Asc(S^*)$-compatible, so
$w_k=w_{k+1}$ implies that conditions (a) and (b) in the definition
of ascent set are both false for this $k$. The falsehood of condition (a)
ensures that $T$ will be a semistandard tableau. On the other hand,
the falsehood of condition (b) guarantees that every cell of $E$
is special for $T$ (not just special for $S$). So $(T,E)\in X$
as needed. 

Knowing that $f$ and $g$ do map into their stated codomains, it is 
now immediate that $f\circ g=\id_Y$ and $g\circ f=\id_{\ssyt^*(\lambda/\mu)}$. 
(This verification does not involve the sets $E$ or the ascent sets; one
only needs the fact that the usual standardization of a semistandard tableau
is reversible if the content word $w$ is known). So the proof is complete.
\end{proof}

\begin{example}
  Applying $f$ to the starred semistandard tableau $T^*$ from
  Example~\ref{ex:sstd-elev} gives $f(T^*)=(S^*,w)$, where
  \begin{equation*}
    \cellsize=3ex
    S^*=\tableau{
    1&2&3&{6}^*&7&{14\!}^*&19&20 \\
    4&5&10&11&21&{23\!}^* \\
    8&9&12&13&22 \\
    15&16&17&18}
\end{equation*}
 and $w=11122223333444555555678$.
  Observe that $w$ is $\Asc(S^*)$-compatible, since  
  $\Asc(S^*)=\{3,7,11,14,20,21,22\}$. Moreover, $g(S^*,w)=T^*$.
\end{example}

\cellsize=2.5ex
\begin{example}
  Applying $g$ to the object $(S^*,w)\in Y$ from
  Example~\ref{ex:std-elev} gives the starred semistandard tableau
  \begin{equation*}
T^*=\tableau{
    1&1&1&{5}^*&5&{9}^* \\
    2&3&{4}^*&{8}^*&8 \\
    3&8\\
    6\\
    7}.
\end{equation*}
 Note that, as required, $T$ is semistandard and $E$ does
  consist of special cells for $T$. Moreover, $f(T^*)=(S^*,w)$.
\end{example}

\begin{example}
  Using Theorem~\ref{thm:skewP-to-FQ}, we can make the following
  calculation.  Each term corresponds to the starred standard
  tableau shown below it:

\begin{alignat*}{7}
  P_{21}(t) = \ \ & \FQ_{21} \quad - \quad && t\FQ_{111} \quad + \quad && \FQ_{12} \quad - \quad && t^2\FQ_{111}. \\
             & \tableau{1&2\\3}	&& \tableau{1&{2}^*\\3} && \tableau{1&3\\2} && \tableau{1&{3}^*\\2}
\end{alignat*}
\end{example}

\begin{remark}[Alternate Formula for $\Mtrans(P,s)$]
As described in~\S\ref{subsec:MPs}, Carbonara~\cite{carbonara}
expresses the entries of the inverse $t$-Kostka matrix $\Mtrans(P,s)$ as 
signed, weighted sums of
  special tournament matrices.  An alternative description can be
  obtained by following $\Mtrans(P,\FQ)$ by the projection from $\QSym$ to
  $\Sym$ given in~\cite{elw}.  The entry of $\Mtrans(P,s)_{\lambda,\mu}$ is
  again described as a sum of signed, weighted objects.  However, in
  this description the objects are pairs $(S^*,T)$ where $S^*\in
  \syt^*(\lambda)$ and $T$ is a ``flat special rim-hook tableau'' of
  shape $\mu$ and content $\Asc'(S^*)$.

  In addition to working for skew Hall-Littlewood polynomials, this
  new description may have computational advantages.  For $n=4$, there
  are 37 special tournament matrices that contribute to the
  calculation of $\Mtrans(P,s)$.  However, only 23 pairs $(S^*,T)$ are now
  needed.  We note that these pairs do not correspond to a subclass of
  special tournament matrices in any simple way.  Carbonara's
  description computes the value $\Mtrans(P,s)_{4,22} = 0$ via the fact that
  there are no special tournament matrices with parameters $\lambda =
  4$ and $\mu = 22$.  There are two such pairs $(S^*,T)$, albeit
  of opposite sign and equal weight.
\end{remark}

\subsection{Expansion of $Q_{\lambda/\mu}$}
\label{subsec:MQFQ-skew}

We recall the definition of the skew Hall-Littlewood polynomials
$Q_{\lambda/\mu}(x_1,\ldots,x_N;t)$~\cite[\S III.5]{Macd}.
For $r\in\N^+$, let $\phi_r(t)=(1-t)(1-t^2)\cdots (1-t^r)$.
For a partition $\lambda$ with $m_i(\lambda)$ parts equal to $i$ for
each $i$, define $b_{\lambda}(t)=\prod_{i\geq 1} \phi_{m_i(\lambda)}(t)$.
For partitions $\mu\subseteq\lambda$, define
$b_{\lambda/\mu}(t)=b_{\lambda}(t)/b_{\mu}(t)$.  Finally, define
\begin{equation*} 
Q_{\lambda/\mu}(x_1,\ldots,x_N;t)
=b_{\lambda/\mu}(t)P_{\lambda/\mu}(x_1,\ldots,x_N;t).
\end{equation*}
From this, Theorem~\ref{thm:skewP-to-FQ} immediately gives 
\[ Q_{\lambda/\mu}(x_1,\ldots,x_N;t)
  =\sum_{S^*\in\syt^*(\lambda/\mu)} \sgn(S^*)t^{\tstat(S^*)}
    b_{\lambda/\mu}(t)\FQ_{\Asc'(S^*)}(x_1,\ldots,x_N). \]

On the other hand, Macdonald~\cite[\S III.5, pp. 227--229]{Macd} gives 
the following monomial expansion of $Q_{\lambda/\mu}$. Suppose
$\nu\subseteq\rho$ are partitions such that $\theta=\rho/\nu$ 
is a horizontal strip. Let $\theta_i'$ be the number of cells
of $\theta$ in column $i$, so $\theta_i'\in\{0,1\}$ since $\theta$
is a horizontal strip.  Define
\begin{equation*}
\phi_{\rho/\nu}(t)=\prod_{i\in I} (1-t^{m_i(\rho)})
\end{equation*} 
where $I=\{i\geq 1: \theta_i'=1\mbox{ and }\theta_{i+1}'=0\}$.  For
$T\in\ssyt_N(\lambda/\mu)$, view $T$ as a nested sequence of
partitions $\mu=\lambda^{(0)}\subseteq\lambda^{(1)}
\subseteq\cdots\subseteq \lambda^{(N)}=\lambda$, where
$\lambda^{(i)}/\lambda^{(i-1)}$ is a horizontal strip consisting of
the cells in $T$ with entry $i$, and let $\phi_T(t)=\prod_{i=1}^N
\phi_{\lambda^{(i)}/\lambda^{(i-1)}}(t)$.  Macdonald's expansion for
$Q_{\lambda/\mu}$ is
\begin{equation}\label{eq:HLQ2}
 Q_{\lambda/\mu}(x_1,\ldots,x_N;t)=\sum_{T\in\ssyt(\lambda/\mu)} \phi_T(t)x^T.
\end{equation}
By imitating the proof we gave in~\S\ref{subsec:MPFQ-skew},
we can use~\eqref{eq:HLQ2} to derive an alternative
fundamental quasisymmetric expansion for $Q_{\lambda/\mu}$
(and hence also for $P_{\lambda/\mu}$).

We merely sketch the necessary changes in the first formula and its proof.  
For a tableau $T\in\ssyt_N(\lambda/\mu)$, define the set of
\emph{$Q$-special cells for $T$} 
by \begin{equation*}
\QSp(T)=\{(i,j)\in\lambda/\mu:\mbox{ for all $u\in\lambda/\mu$, }
 T(u,j+1)\neq T(i,j) \}.
\end{equation*}

So, a cell $c$ is $Q$-special for $T$ iff the entry of $T$ in $c$
does not also appear in the column immediately to the right of $c$.
The \emph{weight} of a $Q$-special cell $(i,j)$ is defined to be
\begin{equation*}
\wt_Q(i,j)=m_j(\lambda^{(T(i,j))})= i-|\{u: (u,j+1)\in\mu/(0)\mbox{ or }
 T(u,j+1)\leq T(i,j)\}|.
\end{equation*}

Now let $\ssyt^{*Q}_N(\lambda/\mu)=\{(T,E):T\in\ssyt_N(\lambda/\mu)
\mbox{ and }E\subseteq\QSp(T)\}$ be the set of \emph{$Q$-starred
semistandard tableaux}.  Since $\phi_T(t)=\prod_{c\in\QSp(T)} (1-t^{\wt_Q(c)})$,
it follows as before that
\begin{equation*}
Q_{\lambda/\mu}(x_1,\ldots,x_N;t)=\sum_{T^*\in\ssyt^{*Q}_N(\lambda/\mu)} 
\sgn(T^*)t^{\tstat_Q(T^*)}x^{T^*},
\end{equation*} 
where $\sgn(T^*)=(-1)^{|E|}$, $\tstat_Q(T^*)=\sum_{c\in E} \wt_Q(c)$,
and $x^{T^*}=x^T$.
Define $Q$-starred standard tableaux in the obvious way.
For $S\in\syt(\lambda/\mu)$, observe that \emph{every} cell in $S$
is $Q$-special for $S$. Given $S^*=(S,E)$ with $E\subseteq\lambda/\mu$
and $k<n=|\lambda/\mu|$, define $k\in\QAsc(S^*)$ iff either 
(a) $k+1$ appears in $S$ in a lower row than $k$; or 
(b) $k\in E$ and $k+1$ appears in $S$ in
the next column after $k$'s column.  Let $\QAsc'(S^*)=\comp(\QAsc(S^*))$.  
Define $Y$ as before, replacing $\Asc(S^*)$ by $\QAsc(S^*)$. 
One should now check that the proof given above adapts
to the present situation without difficulty. We therefore have
the following result.
\begin{theorem}\label{thm:skewQ-to-FQ}
For all skew shapes $\lambda/\mu$, 
\[ Q_{\lambda/\mu}(x_1,\ldots,x_N;t)
  =\sum_{S^*\in\syt^{*Q}(\lambda/\mu)} \sgn(S^*)t^{\tstat_Q(S^*)}
     \FQ_{\QAsc'(S^*)}(x_1,\ldots,x_N). \]
In particular, for all $\lambda\in\Par_n$ and $\alpha\in\Comp_n$,
$\Mtrans(Q,\FQ)_{\lambda,\alpha}=\sum_{S^*} \sgn(S^*)t^{\tstat_Q(S^*)}$
summed over all $Q$-starred standard tableaux having shape $\lambda$
and $Q$-ascent set $\sub(\alpha)$.
\end{theorem}
Dividing through by $b_{\lambda/\mu}(t)$ gives an analogous
$\FQ$-expansion for $P_{\lambda/\mu}$.

\section{Transition Matrices Giving $G$-Expansions}
\label{sec:Gexp}

This section discusses combinatorial formulas for the transition
matrices $\Mtrans(\FQ,G)$, $\Mtrans(M,G)$, $\Mtrans(P,G)$,
$\Mtrans(\mathcal{S},G)$ and $\Mtrans(K,G)$.

\subsection{$\boldsymbol{\Mtrans(\FQ,G)}$}
\label{subsec:MFQG}

Let $\alpha,\beta\in\Comp_n$ with $\beta$ finer than $\alpha$.
Define $\xi_{\alpha,\beta}(j)$ to be $j$ if $\beta_j$ and
$\beta_{j+1}$ are formed from the same part of $\alpha$ and $0$
otherwise. Set $g(\alpha,\beta) = \sum_{j=1}^{\ell(\beta)-1} 
\xi_{\alpha,\beta}(j)$.  

\begin{theorem}\label{thm:MFQG}\label{thm:mfg}
For all $\alpha,\beta\in\Comp_n$,
$\Mtrans(\FQ,G)_{\alpha,\beta}$ is $t^{g(\alpha,\beta)}$ if
$\beta\succeq\alpha$ and $0$ otherwise.
\end{theorem}
\begin{proof}
Since $\Mtrans(\FQ,G)$ is the unique matrix such that $\Mtrans(G,\FQ)\Mtrans(\FQ,G)=I$,
it is enough (by Theorem~\ref{thm:MGFQ}) to show that 
for all compositions $\beta\succeq\alpha$,
  \begin{equation}\label{eq:chi}
    \sum_{\gamma:\, \beta \succeq \gamma \succeq \alpha} 
    (-1)^{\ell(\gamma)-\ell(\alpha)} t^{s(\alpha,\gamma)} t^{g(\gamma,\beta)}
=\delta_{\alpha,\beta}.
  \end{equation}
  Recall that $s(\alpha,\gamma)=\sum_{j=1}^{\ell(\alpha)}
  j(\Bre(\gamma,\alpha)_j-1)$.  If $\alpha = \beta$ then there is a
  single term that is readily seen to equal $1$.  So suppose $\beta
  \succ \alpha$. We define a sign-reversing involution
  $\gamma\mapsto\gamma'$ on the set of compositions $\gamma$ with
  $\beta\succeq\gamma\succeq\alpha$, as follows.  Let $j$ be as small
  as possible such that $\Bre(\beta,\gamma)_j > 1$ or
  $\Bre(\gamma,\alpha)_j > 1$.  Equivalently, this is the smallest $j$
  such that not all of $\alpha_j$, $\beta_j$ and $\gamma_j$ are equal.
  We know that such a $j$ must exist since $\alpha \neq \beta$.

  If $\gamma_j = \beta_j$ (and hence $\beta_j<\alpha_j$), then set
  \begin{equation}\label{eq:join}
    \gamma' = (\gamma_1,\ldots,\gamma_{j-1},\gamma_j+\gamma_{j+1},
    \gamma_{j+2},\ldots,\gamma_{\ell(\gamma)}).
  \end{equation}
  Otherwise, let 
  \begin{equation}\label{eq:split}
    \gamma' = (\gamma_1,\ldots,\gamma_{j-1},\beta_j,\gamma_j-\beta_j,
    \gamma_{j+1},\ldots,\gamma_{\ell(\gamma)}).
  \end{equation}
  If $\gamma'$ is defined according to~\eqref{eq:join}, then
  $s(\alpha,\gamma') = s(\alpha,\gamma) - j$ and $g(\gamma',\beta) =
  g(\gamma,\beta) + j$.  On the other hand, if $\gamma'$ is defined
  according to~\eqref{eq:split}, then $s(\alpha,\gamma') =
  s(\alpha,\gamma) + j$ and $g(\gamma',\beta) = g(\gamma,\beta) - j$.
  It follows immediately that the map $\gamma \mapsto \gamma'$ is a
  sign-reversing, $t$-weight preserving involution on the 
  set of $\gamma$ for which $\beta \succeq \gamma \succeq \alpha$.
  Hence, the sum in~\eqref{eq:chi} is zero as desired.
\end{proof}

\begin{example}
  Let $\alpha = 212135$, $\gamma = 2121323$ and $\beta = 212112212$.
  Note that $\Bre(\gamma,\alpha) = 111112$ and $\Bre(\beta,\gamma) =
  1111212$.  It follows that $s(\alpha,\gamma) = 6\cdot (2-1) = 6$ and
  $g(\gamma,\beta) = 5 + 8 = 13$.  The smallest $j$ for which the
  parts are not all equal is $j=5$.  The parts $\gamma_5$ and
  $\beta_5$ are not equal, so $\gamma'$ gets defined according
  to~\eqref{eq:split}: $\gamma' = 21211223$.  Hence
  $\Bre(\gamma',\alpha) = 111122$ and $\Bre(\beta,\gamma') =
  11111112$.  So $s(\alpha,\gamma') = 5\cdot 1 + 6\cdot 1 = 11$ and
  $g(\gamma',\beta) = 8$.
\end{example}

\begin{example}
Using Theorem~\ref{thm:MFQG}, we calculate 
$\FQ_{3}=G_3+tG_{21}+tG_{12}+t^3G_{111}$,
$\FQ_{21}=G_{21}+tG_{111}$, $\FQ_{12}=G_{12}+t^2G_{111}$,
and $\FQ_{111}=G_{111}$.
\end{example}

\subsection{$\boldsymbol{\Mtrans(M,G)}$}
\label{subsec:MMG}

\begin{theorem}\label{thm:MMG}
For all $\alpha,\beta\in\Comp_n$,
$\Mtrans(M,G)_{\alpha,\beta}$ is
$(-1)^{\ell(\beta)-\ell(\alpha)} 
\prod_{j:\, \xi_{\alpha,\beta}(j) = j} (1-t^j)$
if $\beta\succeq\alpha$ and $0$ otherwise.
\end{theorem}
\begin{proof}
Fix $\alpha,\beta\in\Comp_n$.
 Since $\Mtrans(M,G)=\Mtrans(M,\FQ)\Mtrans(\FQ,G)$, 
  \begin{align*}
    \Mtrans(M,G)_{\alpha,\beta} &= \sum_{\gamma\in\Comp_n}
    \Mtrans(M,\FQ)_{\alpha,\gamma}\Mtrans(\FQ,G)_{\gamma,\beta}\\
    &= \sum_{\gamma: \beta \succeq \gamma \succeq \alpha} 
    (-1)^{\ell(\gamma)-\ell(\alpha)} t^{g(\gamma,\beta)} \\
    &= (-1)^{\ell(\beta)-\ell(\alpha)} \sum_{\gamma: 
      \beta \succeq \gamma \succeq \alpha}
    (-1)^{\ell(\gamma)-\ell(\beta)} t^{g(\gamma,\beta)}.
  \end{align*}
  Any composition $\gamma$ in the above sum is completely determined
  by specifying the values of $\xi_{\gamma,\beta}(j)$ for those $j$
  such that $\xi_{\alpha,\beta}(j)=j$. This follows since
  $\beta\succeq\gamma\succeq\alpha$ and $\xi_{\alpha,\beta}(j)=0$
  force $\xi_{\gamma,\beta}(j)=0$. On the other hand, for each $j$ with 
  $\xi_{\alpha,\beta}(j)=j$, we can always choose $\xi_{\gamma,\beta}(j)$
  to be either $0$ or $j$ when building $\gamma$. Every 
  time we choose to set $\xi_{\gamma,\beta}(j) = j$, we are increasing
  the length difference between $\beta$ and $\gamma$ by $1$.
  Additionally, we are increasing the $t$-weight by $j$.  Since all of
  these choices are independent, we can write 
  \begin{equation*}
    \sum_{\gamma: \beta \succeq \gamma \succeq \alpha} 
  (-1)^{\ell(\gamma)-\ell(\beta)} t^{g(\gamma,\beta)} = 
    \prod_{j:\, \xi_{\alpha,\beta}(j) = j} (1-t^j).
  \end{equation*}
  The theorem follows.
\end{proof}

\begin{example}
  Consider $\alpha = 22$ and $\beta = 1111$.  Then
  $\xi_{\alpha,\beta}(1) = 1$, $\xi_{\alpha,\beta}(2) = 0$ and
  $\xi_{\alpha,\beta}(3) = 3$.  So $\Mtrans(M,G)_{\alpha,\beta} =
  (-1)^2(1-t)(1-t^3)$.
\end{example}

\begin{example}
We calculate $M_3=G_3-(1-t)G_{21}-(1-t)G_{12}+(1-t)(1-t^2)G_{111}$,
 $M_{21}=G_{21}-(1-t)G_{111}$, $M_{12}=G_{12}-(1-t^2)G_{111}$, 
and $M_{111}=G_{111}$.
\end{example}

\subsection{$\boldsymbol{\Mtrans(P,G)}$}
\label{subsec:MPG}

By multiplying $\Mtrans(P,\FQ)$ and $\Mtrans(\FQ,G)$, we obtain the formula 
\begin{equation}\label{eq:P-in-G}
    \Mtrans(P,G)_{\lambda,\beta} = \sum_{\substack{S^*=(S,E)\in\syt^*(\lambda)
     \\ \Asc'(S^*)\preceq\beta}} 
    (-1)^{|E|}t^{\tstat(S^*) + g(\Asc'(S^*),\beta)}.
\end{equation} 

As described in Theorem~\ref{thm:mpg} below, the contribution of each standard
Young tableau $S$ to~\eqref{eq:P-in-G} can be expressed as a product.
One advantage of this latter form is that a frequent group of
cancellations can be identified; see Corollary~\ref{cor:cancel}.  In order
to present these results, we introduce some new notation.

For $S\in \syt(\lambda)$, define $\spec(S)$ and $\wt(c)$ as
in~\S\ref{subsec:MPFQ-skew}.  We define the following subset of
$\spec(S)$:
\begin{equation*}
  \espec(S) = \{c\in\spec(S):\, \Asc(S,\{c\}) \neq \Asc(S,\emptyset)\}.
\end{equation*}
So, $c\in\espec(S)$ if and only if $S(c)-1$ appears in the column
of $S$ just left of the column containing $S(c)$ and $S(c)-1$
appears weakly lower than $S(c)$ does.  The subset $\espec(S)$ keeps
track of which cells actually affect $\Asc'(S^*)$ for starred tableaux
with underlying tableau $S$.  As we are only concerned here with
standard tableaux $S$, we will let $c_j=c_j(S)$ denote the cell of $S$
in which $j$ appears.  Finally, note that the descent set of $S$,
$\Des(S)$, is contained in $\Asc(S^*)$ for any $S^*$ with underlying
tableau $S$.

Our intent is to derive a simplified version of~\eqref{eq:P-in-G}
in which the main sum extends over standard tableaux $S\in\syt(\lambda)$
rather than starred standard tableaux. To obtain $S^*=(S,E)$ from $S$,
we will build the subset $E$ by choosing to include or exclude
various cells $c\in\spec(S)$. The final object $S^*$ is required to
satisfy $\Asc'(S^*)\preceq \beta$, or equivalently, $\Asc(S^*)\subseteq
\sub(\beta)$. By the remark at the end of the last paragraph, this
requirement will be met only if $\Des(S)\subseteq\sub(\beta)$. So
we need only consider standard tableau $S$ satisfying this condition.

The choices of which $c = c_{j+1}(S)\in \spec(S)$ to include in $E$
can be made independently.  We consider each such $c$ according to 
whether $j\in\sub(\beta)$ and whether $c\in \espec(S)$.  We first
note that if $j\not\in\sub(\beta)$ and $c\in \espec(S)$, then
inclusion of $c$ in $E$ would cause $\Asc'(S,E)\not\preceq \beta$.
Hence, this case need not be considered further.  For each of the
remaining possibilities for $j$ and $c$, we consider the net effect on
the signed $t$-weight caused by the decision to include or exclude $c$
from $E$. For each $j\in\sub(\beta)$, let $m_j=m_j(\beta)$ be the
number of elements in $\sub(\beta)$ that are at most $j$.
\begin{enumerate}
\item \emph{$j\not\in\sub(\beta)$ and $c\not\in \espec(S)$}.  If $c$
  is included in $E$, the $t$-weight contribution will be
  $-t^{\wt(c)}$; otherwise the contribution will be $1$.\label{case:i}
\item \emph{$j\in\sub(\beta)$ and $c\in\espec(S)$.}  If $c$ is
  included in $E$, the $t$-weight contribution will be $-t^{\wt(c)}$;
  otherwise it will be $t^{m_j}$.
\item \emph{$j\in\sub(\beta)$ and $c\not\in \espec(S)$.}  Since
  $c\not\in \espec(S)$, the value of $\xi_{\Asc'(S^*),\beta}(j)$
  depends only on whether or not $j\in \Des(S)$.  Hence, there is a
  corresponding $t$-weight contribution of $t^{m_j}$ if and only if $j
  \not\in \Des(S)$.  As in Case~\ref{case:i}, there is an additional
  contribution to the $t$-weight (coming from the $\tstat$ function)
  of $-t^{\wt(c)}$ or $1$ according to whether or not $c$ is in
  $E$.\label{case:iii}
\end{enumerate}

Observing that the factor $1-t^{\wt(c)}$ occurs in both
Case~\ref{case:i} and Case~\ref{case:iii}, we reorganize the cases
slightly to obtain the following.
\begin{theorem}\label{thm:mpg}
For all $\lambda\in\Par_n$ and $\beta\in\Comp_n$,
  \begin{equation}\label{eq:pgprod}
  \Mtrans(P,G)_{\lambda,\beta} = 
  \sum_{\substack{S\in\syt(\lambda)\\ \Des(S)\subseteq \sub(\beta)}}
  \prod_{\substack{j\in\sub(\beta): \\
     c_{j+1}\in \espec(S)}} (t^{m_j}-t^{\wt(c_{j+1})})
  \prod_{\substack{j\in [n-1]:\\c_{j+1}\in \spec(S)\setminus \espec(S)}} 
     t^{m_j'}(1-t^{\wt(c_{j+1})}),
  \end{equation}
  where $m_j' = m_j$ if $j\in \sub(\beta)\setminus \Des(S)$ and $0$ otherwise.
\end{theorem}

\begin{corollary}\label{cor:cancel}
  If $m_j = \wt(c_{j+1})$ for some $j\in\sub(\beta)$ with $c_{j+1}\in
  \espec(S)$, then $S$ can be omitted from the sum in~\eqref{eq:pgprod}.
\end{corollary}

\begin{example}
  Let $\lambda = 32$ and $\beta = 1211$.  Note that $\sub(\beta) =
  \{1,3,4\}$ and so $m_1=1$, $m_3=2$, and $m_4=3$.  In
  Table~\ref{tab:ptogex} we list the five elements of $\syt(32)$
  (referred to from left to right as $S_1,\ldots,S_5$) along with
  pertinent data.  The row labeled $\prod_1$ (resp. $\prod_2$) gives
  the contributions from the first (resp. second) product
  in~\eqref{eq:pgprod}.  Since $\Des(S_2), \Des(S_3)\not\subseteq
  \sub(\beta)$, $\prod_1$ and $\prod_2$ have been left blank for these
  two tableaux.  (For reference, the corresponding products for these
  tableaux are $(t-t)\cdot t^2(1-t)(1-t)$ and
  $(t-t)(t^2-t)(t^3-t^2)\cdot 1$, respectively.)  Note that
  Corollary~\ref{cor:cancel} applies to $S_1$ with $j=m_j=1$.  So the
  only contributions are from the last two columns and we find that
  \begin{equation*} 
    \Mtrans(P,G)_{32,1211} = (t^2-t)(1-t) + (t^3-t^2)(1-t) = -t^4 + t^3 + t^2 - t.
  \end{equation*}
\end{example}

\begin{table}[h]
    \centering
    \caption{Computation of $\Mtrans(P,G)_{32,1211}$.}
  \begin{tabular}{@{}cccccc@{}} \toprule
    $S$ & $\young(123,45)$ & $\young(124,35)$ & $\young(125,34)$ 
    & $\young(134,25)$ & $\young(135,24)$\\ \midrule
    $\Des(S)$ & $\{3\}$ & $\{2,4\}$ & $\{2\}$ & $\{1,4\}$ & $\{1,3\}$\\
    $\spec(S)$ & $\{c_2,c_3,c_5\}$ & $\{c_2,c_4,c_5\}$ & 
                 $\{c_2,c_4,c_5\}$ & $\{c_3,c_4,c_5\}$ & $\{c_3,c_4,c_5\}$\\
    $\espec(S)$ & $\{c_2,c_3,c_5\}$ & $\{c_2\}$ & $\{c_2,c_4,c_5\}$ 
      & $\{c_3,c_4\}$ & $\{c_3,c_5\}$\\ 
    $\prod_1$ & $(t-t)(t^3-t)$ & & &
    $(t^2-t)$ & $(t^3-t^2)$\\
    $\prod_2$ & $1$ & & & $(1-t)$ & $(1-t)$\\\bottomrule
  \end{tabular}
  \label{tab:ptogex}
\end{table}

\subsection{$\boldsymbol{\Mtrans(\mathcal{S},G)}$}
\label{subsec:MSG}
For an $\SSCT$ $T$, define the \emph{ascent set}, denoted $\Asc(T)$, to be
the set of values $i$ for which the leftmost occurrence of $i$ appears
strictly to the right of the rightmost occurrence of $i+1$. If
$\Asc(T) = \{a_1, a_2, \ldots \}$, we let $\SumAsc(T) = a_1+a_2+
\cdots$. For example, the $\SSCT$
\begin{equation*}
  \Yvcentermath1 T=\young(22211,43,7555331,8866)
\end{equation*} 
has ascent set $\Asc(T) = \{1,3,6\}$ and $\SumAsc(T) = 1+3+6=10$.

\begin{theorem}\label{thm:MSG}
For all $\alpha, \beta \in \Comp_n$,
$\Mtrans(\mathcal{S},G)_{\alpha,\beta} 
= \sum_{T \in \SSCT(\alpha,\beta)} t^{\SumAsc(T)}$.
\end{theorem}
\begin{proof}
  Fix $\alpha, \beta \in \Comp_n$.  We proceed to justify the
  following string of equalities:
  \begin{align*}
    \Mtrans(\mathcal{S},G)_{\alpha,\beta}
    &= \sum_{\gamma \in \Comp_n} \Mtrans(\mathcal{S},F)_{\alpha,\gamma} \Mtrans(F,G)_{\gamma,\beta} \\
    &= \sum_{\gamma: 
\gamma \preceq \beta} |\SCT(\alpha,\gamma)| t^{g(\gamma,\beta)} \\
    &= \sum_{T' \in \SSCT(\alpha,\beta)} t^{\SumAsc(T')}.
\end{align*}
The first equality follows from the fact that $\Mtrans(\mathcal{S},G) =
\Mtrans(\mathcal{S},F)\Mtrans(F,G)$.  The second follows from
Theorems~\eqref{thm:msf} and~\eqref{thm:mfg}.  To prove the third, we
construct a bijection 
\begin{equation*}
f=f_{\alpha,\beta}:\bigcup_{\gamma:\gamma\preceq\beta} \SCT(\alpha,\gamma)
 \rightarrow \SSCT(\alpha,\beta)
\end{equation*}
such that for all $T\in\SCT(\alpha,\gamma)$ with $\gamma\preceq\beta$,
$\SumAsc(f(T))=g(\gamma,\beta)$.
Fix $T$ in the domain of $f$, say $T\in\SCT(\alpha,\gamma)$
with $\gamma\preceq\beta$.
For $0\leq i\leq\ell(\gamma)$, let $d_i = \gamma_1 + \cdots + \gamma_i$, so
$\{d_0,\ldots,d_{\ell(\gamma)}\} = \Des(T) \cup \{ 0,n \}$. Then for
$1\leq i\leq\ell(\gamma)$, the entries $d_{i-1}+1, \ldots, d_i$, when
read from right to left in $T$, appear in increasing order. Let $A_i$
be the set of cells occupied by these values. Since $\beta$ is finer
than $\gamma$, there exist indices $0 = i_0 < i_1 < \cdots <
i_{\ell(\gamma)} = \ell(\beta)$ such that $\gamma_j =\beta_{i_{j-1}+1}
+ \cdots + \beta_{i_j}$ for all $j$.  Construct $T'=f(T)$ by
labeling the cells in each set $A_j$ from right to left by the sequence
\begin{equation*}
  \underbrace{(i_{j-1}+1) \cdots
  (i_{j-1}+1)}_{\beta_{i_{j-1}+1}} \underbrace{(i_{j-1}+2) \cdots
  (i_{j-1}+2)}_{\beta_{i_{j-1}+2}} \cdots \underbrace{i_j \cdots
  i_j}_{\beta_{i_j}};
\end{equation*} 
A simple check confirms that $T'$ satisfies (C1), (C2), and (C3) in
the definition of an $\SSCT$ and has content $\beta$,
so $f$ does map into the codomain $\SSCT(\alpha,\beta)$.
Furthermore, note that the ascents in
$T'$ occur only inside the blocks $A_j$.  Each such place where the
value changes from $k$ to $k+1$ corresponds to the inclusion of $k$ in
$\Asc(T')$ and hence the contribution of $k$ to
$\SumAsc(T')$. Furthermore, such a change only happens when the parts
$\beta_k$ and $\beta_{k+1}$ belong to the same part of $\gamma$.  In
this case, we have $\xi_{\gamma,\beta}(k) = k$ and hence there will
be a contribution of $k$ to $g(\gamma,\beta)$. Thus $\SumAsc(T') =
g(\gamma,\beta)$. Since this relabeling is reversible, $T$ can
be reconstructed from $T'=f(T)$, so $f$ is a bijection.  See
Example~\ref{ex:MSG} for an illustration of the bijection.
\end{proof}

\begin{example}\label{ex:MSG}
  Let $\alpha = 5274$, $\beta = 33313212$, $\gamma = 64332$, and
\begin{equation*}
  \Yvcentermath1 T=\young(65432,{\ten}9,{\sixteen}{\thirteen}{\twelve}{\eleven}871,{\eighteen}{\seventeen}{\fifteen}{\fourteen})\,.
\end{equation*}
First note that $g(\gamma,\beta) = 1+3+6=10$, since the $i$ for which
$\beta_i$ and $\beta_{i+1}$ came from the same part of $\gamma$ are
precisely $1$, $3$ and $6$. In order to construct $T'$, note that the
first part of $\gamma$ is a $6$, which is refined by the first two
parts of $\beta$ (both equal to $3$). Thus, we first need to look at the
horizontal strip of cells containing the letters $1,2,3,4,5,6$ 
and change these letters to $1,1,1,2,2,2$, respectively. The
other conversions, based on the parts of $\beta$ that fit into each
part of $\gamma$, are as follows:
\begin{align*}
7,8,9,10 		& \mapsto 3,3,3,4 \\
11,12,13 		& \mapsto 5,5,5 \\
14,15,16	 	& \mapsto 6,6,7 \\
17,18		& \mapsto 8,8.
\end{align*}
Thus, the resulting $T'$ with content $\beta$ is
\begin{equation*}
\Yvcentermath1 T'=\young(22211,43,7555331,8866)\,.
\end{equation*}
Note that $\Asc(T')=\{1,3,6\}$, so $\SumAsc(T')=1+3+6=10=g(\gamma,\beta)$.
\end{example}

\begin{example}
Using Theorem \ref{thm:MSG}, the $G$-expansion of $\mathcal{S}_{13}(t)$ is
\begin{alignat*}{15}
  & G_{13} & + & \ \ G_{22} & + & \ \ t^2 G_{211} & + & \ \ t^2 G_{121} & + & 
  \ \ t G_{112} & + & \ \ t^2 G_{112} & + & \ \ t^4 G_{1111} & + & \ \ t^5 G_{1111}. \\
  & \tableau{1\\2&2&2} & & \ \tableau{1\\2&2&1} & & \ \ 
  \tableau{1\\3&2&1} & & \ \ \tableau{1\\3&2&2} & & \ \
  \tableau{2\\3&3&1} & & \ \ \tableau{1\\3&3&2} & & \ \ 
  \tableau{2\\4&3&1} & & \ \ \tableau{1\\4&3&2} 
\end{alignat*}
\end{example}

\subsection{$\boldsymbol{\Mtrans(K,G)}$}
\label{subsec:MKG}

Consider compositions $\alpha\in\Comp'_n$ and $\beta\in\Comp_n$ such that
$\sub(\alpha) \subseteq \sub(\beta) \cup (\sub(\beta)+1)$.  Let
$\sub(\alpha) = \{a_1 < \cdots < a_m\}$ and $\sub(\beta) = \{b_1 <
\cdots < b_p\}$.  For $A \subseteq [n-1]$, let $A-1$ denote the set
obtained by subtracting one from every element in $A$ and removing $0$ if
it appears.  Define the polynomial $k(\alpha, \beta)$ as follows.  
For all $b_i\in\sub(\beta)$, let
\begin{equation}\label{eq:cases}
  k(b_i) = 
  \begin{cases}	
    1+t^i, & \textrm{if } b_i \notin \sub(\alpha) \cup (\sub(\alpha) - 1); \\
    t^{i-1}+t^i, & \textrm{if } b_i \in \sub(\alpha) 
        \textrm{ and } b_i-1 = b_{i-1};\\
    1, & \textrm{otherwise.}
  \end{cases}
\end{equation}
Finally, define $k(\alpha,\beta) = \prod_{i=1}^p k(b_i)$.

\begin{example}\label{ex:uglyexample}
If $\alpha = 4253$ and $\beta = 12213221$, 
we have $\sub(\alpha) = \{4,6,11\}$, $\sub(\alpha)-1 = \{3,5,10\}$, and
$\sub(\beta)=\{1,3,5,6,9,11,13\}$. Thus, $k(\alpha,\beta) = 
(1+t) \cdot (t^3+t^4) \cdot (1+t^5) \cdot (1+t^7)$.

Here is a graphical way of looking at this example. 
Depict a composition $\alpha = (\alpha_1, \ldots, \alpha_m)$ of $n$ 
on the interval $[0,n]$ by highlighting the points in $\sub(\alpha)$. 
As shown in Figure~\ref{fig:uglyexample}, highlight the points in $[0,n]$
between an element $a-1 \in \sub(\alpha)-1$ and 
$a \in \sub(\alpha)$. No two of the highlighted subintervals share
an endpoint, because of the assumption that $\alpha\in\Comp'_n$.
Now, the three possible cases in~\eqref{eq:cases} turn into the following:
\begin{enumerate}
\item $b_i$ does not touch any of the highlighted intervals 
(e.g., $b_1=1$, $b_5=9$, and $b_7=13$);
\item $b_{i-1}$ and $b_{i}$ are the two ends of a highlighted interval 
(e.g., $b_3=5$ and $b_4=6$); or
\item $b_i$ is an end of a highlighted interval, but $b_{i-1}$ 
is not at the other end of the highlighted interval
(e.g., $b_2=3$ and $b_6=11$).
\end{enumerate}
\end{example}

\begin{figure}[h] 
$\begin{array}{c}
\begin{picture}(100,100)(100,-100)
\put(-60,0){\makebox{$\sub(\alpha)$}}
\put(0,10){\makebox{$0$}}
\put(80,10){\makebox{$4$}}
\put(120,10){\makebox{$6$}}
\put(218,10){\makebox{$11$}}
\put(278,10){\makebox{$14$}}
\put(0,0){\makebox{$\circ$}}
\put(80,0){\makebox{$\bullet$}}
\put(120,0){\makebox{$\bullet$}}
\put(220,0){\makebox{$\bullet$}}
\put(280,0){\makebox{$\circ$}}
\put(4.5,3){\line(1,0){276.5}}

\put(-60,-30){\makebox{$\sub(\alpha)-1$}}
\put(0,-20){\makebox{$0$}}
\put(60,-20){\makebox{$3$}}
\put(100,-20){\makebox{$5$}}
\put(198,-20){\makebox{$10$}}
\put(278,-20){\makebox{$14$}}
\put(0,-30){\makebox{$\circ$}}
\put(60,-30){\makebox{$\bullet$}}
\put(100,-30){\makebox{$\bullet$}}
\put(200,-30){\makebox{$\bullet$}}
\put(280,-30){\makebox{$\circ$}}
\put(4.5,-27){\line(1,0){276.5}}

\put(0,-50){\makebox{$0$}}
\put(60,-50){\makebox{$3$}}
\put(100,-50){\makebox{$5$}}
\put(198,-50){\makebox{$10$}}
\put(278,-50){\makebox{$14$}}
\put(80,-50){\makebox{$4$}}
\put(120,-50){\makebox{$6$}}
\put(218,-50){\makebox{$11$}}
\put(0,-60){\makebox{$\circ$}}
\put(60,-60){\makebox{$\bullet$}}
\put(100,-60){\makebox{$\bullet$}}
\put(200,-60){\makebox{$\bullet$}}
\put(80,-60){\makebox{$\bullet$}}
\put(120,-60){\makebox{$\bullet$}}
\put(220,-60){\makebox{$\bullet$}}
\put(280,-60){\makebox{$\circ$}}
\put(63,-57){\linethickness{1mm}\line(1,0){20}}
\put(103,-57){\linethickness{1mm}\line(1,0){20}}
\put(203,-57){\linethickness{1mm}\line(1,0){20}}
\put(4.5,-57){\line(1,0){276.5}}

\put(-60,-90){\makebox{$\sub(\beta)$}}
\put(0,-80){\makebox{$0$}}
\put(20,-80){\makebox{$1$}}
\put(60,-80){\makebox{$3$}}
\put(100,-80){\makebox{$5$}}
\put(120,-80){\makebox{$6$}}
\put(180,-80){\makebox{$9$}}
\put(218,-80){\makebox{$11$}}
\put(258,-80){\makebox{$13$}}
\put(278,-80){\makebox{$14$}}
\put(0,-90){\makebox{$\circ$}}
\put(20,-90){\makebox{$\bullet$}}
\put(60,-90){\makebox{$\bullet$}}
\put(80,-90){\makebox{$\bullet$}} 
\put(100,-90){\makebox{$\bullet$}}
\put(120,-90){\makebox{$\bullet$}}
\put(180,-90){\makebox{$\bullet$}}
\put(200,-90){\makebox{$\bullet$}}
\put(220,-90){\makebox{$\bullet$}}
\put(260,-90){\makebox{$\bullet$}}
\put(280,-90){\makebox{$\circ$}}
\put(63,-87){\linethickness{1mm}\line(1,0){20}}
\put(103,-87){\linethickness{1mm}\line(1,0){20}}
\put(203,-87){\linethickness{1mm}\line(1,0){20}}
\put(4.5,-87){\line(1,0){276.5}}
\put(10,-100){\makebox{$1+t$}}
\put(60,-100){\makebox{$1$}}
\put(100,-100){\makebox{$t^3+t^4$}}
\put(170,-100){\makebox{$1+t^5$}}
\put(220,-100){\makebox{$1$}}
\put(250,-100){\makebox{$1+t^7$}}

\end{picture} \\
\end{array}$
\caption{For $\sub(\alpha)=\{4,6,11\}$ 
and $\sub(\beta) = \{1,3,5,6,9,11,13\}$, we have $k(\alpha,\beta) = (1+t)
\cdot (t^3+t^4) \cdot (1+t^5) \cdot (1+t^7)$.}
\label{fig:uglyexample}
\end{figure}
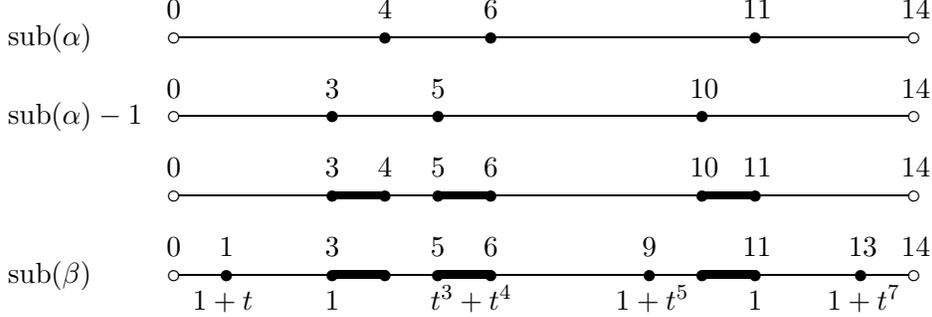

\begin{theorem}\label{thm:MKG}
For $\alpha\in\Comp'_n$ and $\beta\in\Comp_n$,
$\Mtrans(K,G)_{\alpha,\beta}$ is  
$k(\alpha,\beta)$ if $\sub(\alpha)\subseteq 
\sub(\beta) \cup (\sub(\beta)+1)$ and $0$ otherwise.
\end{theorem}
\begin{proof}
Fix $\alpha \in \Comp'_n$ and $\beta \in \Comp_n$. 
Since $\Mtrans(K,G) = \Mtrans(K,F)\Mtrans(F,G)$,
\begin{align*}
  \Mtrans(K,G)_{\alpha,\beta} &= \sum_{\gamma \in \Comp_n} \Mtrans(K,F)_{\alpha,\gamma} \Mtrans(F,G)_{\gamma,\beta} \\
  &= \sum_{\substack{\gamma:\  
\sub(\alpha) \subseteq \sub(\gamma) \bigtriangleup 
(\sub(\gamma) + 1)\\ \text{and }\beta\succeq\gamma}} t^{g(\gamma,\beta)}.
\end{align*} 
We now describe a sequence of choices that will construct
exactly the compositions $\gamma$ indexing the sum above.
By examining how each choice contributes to 
$g(\gamma,\beta)=\sum_j \xi_{\gamma,\beta}(j)$
in the power of $t$, we will obtain the product $k(\alpha,\beta)$.
For convenience of notation, let $A=\sub(\alpha)=\{a_1<\cdots<a_m\}$ and
$B=\sub(\beta)=\{b_1<\cdots<b_l\}$ as above.  Rather than
constructing $\gamma$ itself, we make choices concerning the
members of the set $C=\sub(\gamma)$. The requirements
$\sub(\alpha)\subseteq\sub(\gamma)\bigtriangleup(\sub(\gamma)+1)$
and $\beta\succeq\gamma$ translate into the requirements
$A\subseteq C\bigtriangleup (C+1)$ and $C\subseteq B$.
Since we must have $C\subseteq B$, we need only decide,
for each $b_i\in B$, whether to include or exclude $b_i$ from the set $C$.
When making these decisions, it is required that each
$a_j\in A$ belong to exactly one of the sets $C$ or $(C+1)$
being built. Equivalently, for each $a_j\in A$, exactly
one of the numbers $a_j$ or $a_j-1$ must be chosen for inclusion in $C$.  
Since $C\subseteq B$, this requirement on $A$
will only be possible if $A\subseteq B\cup (B+1)$.
Thus if $A\not\subseteq B\cup (B+1)$, we must have $\Mtrans(K,G)_{\alpha,\beta}=0$.
From now on, assume $A\subseteq B\cup (B+1)$.  
As in~\eqref{eq:cases} and Example~\ref{ex:uglyexample},
there are three cases to consider.

\begin{itemize}
\item Case~1: $b_i\not\in A\cup (A-1)$ (i.e., $b_i$ does not touch any
of the highlighted intervals in the picture). In this case,
we may freely choose to include $b_i$ as a member of $C$ or not,
which will not affect the validity of the condition 
$A\subseteq C\bigtriangleup (C+1)$. Including $b_i$ in $C=\sub(\gamma)$
will cause $\beta_i$ and $\beta_{i+1}$ to come from different parts
of $\gamma$, so that $\xi_{\gamma,\beta}(i)=0$ for this choice.
Excluding $b_i$ from $C$ will cause $\beta_i$ and $\beta_{i+1}$ to come
from the same part of $\gamma$, so that $\xi_{\gamma,\beta}(i)=i$
for this choice. The net contribution to $t^{g(\gamma,\beta)}$
based on the choice involving this $b_i$ is therefore $(1+t^i)=k(b_i)$.  

\item Case~2: $b_i\in A$ and $b_{i-1}=b_i-1$
(i.e., $b_{i-1}$ and $b_i$ are the two ends of a highlighted interval
in the picture). In this case, note that $b_i$ equals some $a_j$,
and $b_{i-1}=a_j-1$, so the requirement on $A$ forces us to
include exactly one of $b_{i-1}$ or $b_i$ in $C$.
If we decide that $b_{i-1}\not\in C$ and $b_i\in C$,
 we see (as in Case~1) that $\xi_{\gamma,\beta}(i-1)=i-1$
 and $\xi_{\gamma,\beta}(i)=0$.
If, instead, we decide that $b_{i-1}\in C$ and $b_i\not\in C$,
 we have $\xi_{\gamma,\beta}(i-1)=0$
 and $\xi_{\gamma,\beta}(i)=i$.
The net contribution to $t^{g(\gamma,\beta)}$ based on the
 choice involving both $b_{i-1}$ and $b_i$ is therefore
 $(t^{i-1}+t^i)=k(b_i)$.  

\item Case~3: The first two cases do not occur (i.e., $b_i$ is one end
of a highlighted interval in the picture, but $b_{i-1}$ is not
at the other end of this interval). Consider several subcases.
 First, suppose $b_i=a_j-1$ for some $a_j\in A$, and $a_j\not\in B$
 (i.e., $b_i$ is the left end of an interval, and the right
  end of this interval is not in $B$). The requirement on $A$
 forces us to include $b_i$ in $C$, giving $\xi_{\gamma,\beta}(i)=0$
 for this $i$ and a contribution of $k(b_i)=1$ to $t^{g(\gamma,\beta)}$.
 Second, suppose $b_i=a_j-1$ for some $a_j\in A$, and 
 $a_j=b_{i+1}$ is in $B$ as well (i.e., $b_i$ is the left end
 of an interval, and the right end $b_i+1=b_{i+1}$ is also in $B$). 
 Then we will decide the membership of $b_i$ in $C$ 
 as part of the Case~2 analysis for index $i+1$. Since the contribution
 of this $b_i$ to $t^{g(\gamma,\beta)}$ has already been accounted for
 elsewhere, we multiply by $k(b_i)=1$ in this subcase.
 Third, suppose $b_i=a_j$ for some $a_j\in A$, and
 $a_j-1$ is not in $B$ (i.e., $b_i$ is the right end of an interval
  whose left end is not in $B$). The requirement on $A$
 forces us to include $b_i$ in $C$, giving a contribution
 of $k(b_i)=1$ to $t^{g(\gamma,\beta)}$.  
\end{itemize}
 Combining all the choices using the product rule 
for weighted sets, we see that
\[ \Mtrans(K,G)_{\alpha,\beta}
=\sum_{C: A\subseteq C\bigtriangleup (C+1), C\subseteq B}
   t^{g(\comp(C),\comp(B))} = \prod_{i=1}^p k(b_i)=k(\alpha,\beta) \]
when $A\subseteq B\cup (B+1)$.  
\end{proof}

\begin{example}
\begin{equation*}
  K_{22} = G_{22} + G_{13} + (1+t^2)G_{211} + 
  (1+t^2)G_{121} + (t+t^2)G_{112} + (t+t^2)(1+t^3)G_{1111}.
\end{equation*}
\end{example}

\section{Transformed Hall-Littlewood Polynomials}
\label{sec:Qpleth}

We conclude the paper with a brief discussion of transition
matrices involving plethystically transformed Hall-Littlewood polynomials.

\subsection{Definition of $Q'_{\lambda}$}
\label{subsec:Q'-def}

For any $\lambda\in\Par$, recall that $Q_{\lambda}(x;t)$
is the Hall-Littlewood polynomial obtained from $P_{\lambda}(x;t)$
by the formula $Q_{\lambda}(x;t)=b_{\lambda}(t)P_{\lambda}(x;t)$
(see~\S\ref{subsec:MQFQ-skew}). Define the \emph{transformed
Hall-Littlewood symmetric function} by setting
\[ Q'_{\lambda}(x;t)=Q_{\lambda}\left[\frac{x}{1-t};t\right]
  =Q_{\lambda}(x_it^j:i\geq 1,j\geq 0); \]
in other words, we replace variables $x_1,x_2,\ldots$ 
in $Q_{\lambda}$ by the monomials $x_it^j$ for all $i\geq 1$ and $j\geq 0$.
We can also define $Q'_{\lambda}$ using plethystic substitution.
Recall that the power-sums $p_k$ are algebraically independent generators
of the algebra $\Sym$. Then $Q'_{\lambda}=\phi(Q_{\lambda})$, 
where $\phi$ is the unique $\field$-algebra homomorphism on $\Sym$
sending each $p_k$ to $p_k/(1-t^k)$. More information on
the transformed Hall-Littlewood polynomials can be found in
Macdonald's book~\cite{Macd} and in the 
papers~\cite{DLT-HL-SLC,garsia-orthog-milne,kirillov-HL,warnaar-HL}. 

\subsection{The Transformed Macdonald Polynomials $H_{\mu}$}
\label{subsec:Hmu}

One can also obtain $Q_{\lambda}'$ as the $q=0$ specialization of the
\emph{transformed Macdonald polynomials} $H_{\lambda}(x;q,t)$.  For a
partition $\lambda = (\lambda_1,\lambda_2,\ldots,\lambda_k)$, set
$n(\lambda) = \sum_{i=0}^k (i-1)\lambda_i$.  Haglund~\cite{hag} found
an explicit combinatorial formula for the fundamental quasisymmetric
expansion of the modified Macdonald polynomials
$\tilde{H}_{\lambda}(x;q,t)=t^{n(\lambda)}H_{\lambda}(x;q,1/t)$, which
was proved by Haglund, Haiman, and Loehr in~\cite{HHL}.  This formula
is readily adjusted to yield a formula for $H_{\lambda}$ of the form
\[ H_{\lambda}(x;q,t)=\sum_{w\in S_n}
     t^{\comaj_{\lambda}(w)}q^{\inv_{\lambda}(w)}\FQ_{\comp(\IDes(w))}, \]
where $\lambda\in\Par_n$, $\IDes(w)$ is the sum of all $i<n$ such
that $i+1$ appears to the left of $i$ in $w_1\cdots w_n$,
and $\comaj_{\lambda}$ and $\inv_{\lambda}$ are certain permutation
statistics depending on $\lambda$.

We define the statistics $\comaj_{\lambda}(w)$ and $\inv_{\lambda}(w)$ 
using the specific example $n=9$, $\lambda=(4,3,2)$, and $w=428165793$.
Fill the cells in the diagram of $\lambda$ with the letters in $w$,
working from the shortest row up to the longest row, filling each
row from left to right, as shown here:
\[ \tableau{5 & 7 & 9 & 3 \\ 8 & 1 & 6 \\ 4 & 2}. \]
This filled diagram determines a list of \emph{column words}, obtained
by reading the entries in each column from bottom to top.
For a word $v=v_1\cdots v_k$, let $\comaj(v)$ be the sum of all
$i<k$ such that $v_i<v_{i+1}$. Let $\comaj_{\lambda}(w)$ be the sum
of $\comaj(v)$ over all column words $v$ in the filled diagram for $w$.
In our example, 
\[ \comaj_{\lambda}(w)=\comaj(485)+\comaj(217)+\comaj(69)+\comaj(3)
                  =1+2+1+0=4. \]
Next, $\inv_{\lambda}(w)$ counts the \emph{inversion triples} in the
filled diagram for $w$, which are defined as follows. Consider
three cells in the diagram positioned as shown here:
\[ \tableau{ {b} \\ {a} & & & {c} }. \]
So $b$ is directly above $a$ and $c$ occurs somewhere to the right of $a$
in the same row. We also allow $a$ and $c$ to be in the top row of
the diagram, in which case we take $b=\infty$.
This triple of cells is an inversion triple iff $a<b<c$ or $b<c<a$ or $c<a<b$.
In our example, the inversion triples $(a,b,c)$ are
$(5,\infty,3)$, $(7,\infty,3)$, $(9,\infty,3)$, $(8,5,6)$, and $(4,8,2)$,  so 
$\inv_{\lambda}(w)=5$.  Since $\IDes(w)=\{1,3,5,7\}$, $w$ contributes a monomial
$t^4q^5$ to the coefficient of $\FQ_{12222}$ in $H_{432}(x;q,t)$.

\begin{theorem}[\cite{HHL}]\label{thm:MHFQ}
For all $\lambda\in\Par_n$ and $\alpha\in\Comp_n$, 
$\Mtrans(H,\FQ)_{\lambda,\alpha}$ is 
$\sum_w t^{\comaj_{\lambda}(w)}q^{\inv_{\lambda}(w)}$
summed over all $w\in S_n$ with $\IDes(w)=\sub(\alpha)$.  
\end{theorem}

\subsection{$\boldsymbol{\Mtrans(Q',\FQ)}$}
\label{subsec:MQ'FQ}

The transformed Hall-Littlewood polynomials are obtained from the
transformed Macdonald polynomials by setting $q=0$, i.e.,
$Q'_{\lambda}(x;t)=H_{\lambda}(x;0,t)$. Using Theorem~\ref{thm:MHFQ},
we therefore obtain the following combinatorial description of
the fundamental quasisymmetric expansion of $Q'$.

\begin{theorem}\label{thm:MQ'FQ1}
For all $\lambda\in\Par_n$ and $\alpha\in\Comp_n$, 
$\Mtrans(Q',\FQ)_{\lambda,\alpha}$ is $\sum_w t^{\comaj_{\lambda}(w)}$
summed over all $w\in S_n$ with $\IDes(w)=\sub(\alpha)$ 
and $\inv_{\lambda}(w)=0$.  
\end{theorem}

On the other hand, it follows from~\cite[p. 241]{Macd} that
$\Mtrans(Q',s)=M(s,P)^T$, the transpose of the $t$-Kostka matrix.  
Accordingly, $\Mtrans(Q',\FQ)=\Mtrans(s,P)^T\Mtrans(s,\FQ)$.
Carrying out the matrix multiplication and
using the interpretations of the latter matrices
from~\S\ref{subsec:MsP} and~\S\ref{subsec:MsFQ}, we obtain a second
combinatorial interpretation of $\Mtrans(Q',\FQ)$:

\begin{theorem}\label{thm:MQ'FQ2}
For all $\lambda\in\Par_n$ and $\alpha\in\Comp_n$, 
$\Mtrans(Q',\FQ)_{\lambda,\alpha}$ is $\sum_{S,T} t^{\chg(T)}$ summed 
over all pairs $(S,T)$ where $S$ is a standard tableau with
$\Des(S)=\sub(\alpha)$, $T$ is a semistandard tableau with content $\lambda$,
and $S$ and $T$ have the same shape $\mu\in\Par_n$.  
\end{theorem}

The objects appearing in this theorem may remind the reader of
the Robinson-Schensted-Knuth (RSK) algorithm. Indeed, the connection
between Theorem~\ref{thm:MQ'FQ1} and Theorem~\ref{thm:MQ'FQ2}
is explained in~\cite[\S7]{HHL} by a careful analysis involving 
the RSK algorithm.

\subsection{Open Problems}
\label{subsec:open-prob}

By matrix multiplication, we obtain a formula
for $\Mtrans(Q',G)_{\lambda,\alpha}=\sum_{\beta\in\Comp_n} 
\Mtrans(Q',\FQ)_{\lambda,\beta}\Mtrans(\FQ,G)_{\beta,\alpha}$.
We leave it as an open problem to find a more direct 
combinatorial interpretation of this transition matrix.

Another open question is to define an appropriate notion
of the ``plethystically transformed'' Hivert quasisymmetric
function $G'_{\alpha}(x;t)$, which should play a role in
$\QSym$ analogous to the role of $Q'_{\lambda}$ in $\Sym$.
It would be interesting to study the combinatorics of transition
matrices to and from such a basis.

Finally, one could develop transition matrices between Hivert's
$H$-basis (see Remark~\ref{rem:nsym}) and other known bases of
$\NSym$, for example, those constructed in \cite{bbssz, bz, gkllrt,
  hlmvw, tevlin}.

\section{Appendix: Examples of Transition Matrices}
\label{sec:app}

This appendix lists specific examples of transition matrices (old and new)
discussed in this paper. In each case, we give the relevant 
matrix for $n=4$ (many computed using SAGE~\cite{sage,sage-code}).

\[ \Mtrans(s,m)=\bordermatrix{%
& 4 & 31 & 22 & 211 & 1111 \cr
4& 1 & 1 & 1 & 1 & 1 \cr
31& 0 & 1 & 1 & 2 & 3 \cr
22& 0 & 0 & 1 & 1 & 2 \cr
211& 0 & 0 & 0 & 1 & 3 \cr
1111& 0 & 0 & 0 & 0 & 1 }. \]
\medskip
\[ \Mtrans(m,s)=\bordermatrix{
& 4 & 31 & 22 & 211 & 1111 \cr
4& 1 & -1 & 0 & 1 & -1 \cr
31& 0 &  1 &-1 &-1 & 2  \cr
22& 0 &  0 & 1 &-1 & 1 \cr
211& 0 &  0 & 0 & 1 &-3 \cr
1111& 0 &  0 & 0 & 0 & 1 }. \]
\medskip
\[ \Mtrans(s,P)=\bordermatrix{
& 4 & 31 & 22 & 211 & 1111 \cr
4 & 1 & t & t^2 & t^3 & t^6 \cr
31& 0 & 1 & t & t+t^2 & t^3+t^4+t^5 \cr
22& 0 & 0 & 1 & t & t^2+t^4 \cr
211& 0 &0 & 0& 1 & t+t^2+t^3 \cr
1111&0 &0 &0 &0 &1 }=\Mtrans(Q',s)^T. \]
\medskip
\[ \Mtrans(P,s)=\bordermatrix{
& 4 & 31 & 22 & 211 & 1111 \cr
4 & 1 & -t & 0 & t^2 & -t^3 \cr
31& 0 & 1 & -t & -t & t^2+t^3 \cr
22& 0 & 0 & 1 & -t & t^3 \cr
211& 0 &0 & 0& 1 & -t-t^2-t^3 \cr
1111&0 &0 &0 &0 &1 }=\Mtrans(s,Q')^T. \]
\medskip
\[ \Mtrans(P,m)=\bordermatrix{%
& 4 & 31 & 22 & 211 & 1111 \cr
4& 1 & 1-t & 1-t & (1-t)^2 & (1-t)^3 \cr
31& 0 & 1 & 1-t & 2(1-t) & 3-5t+t^2+t^3 \cr
22& 0 & 0 & 1 & 1-t & 2-3t+t^3 \cr
211& 0 & 0 & 0 & 1 & 3-t-t^2-t^3 \cr
1111& 0 & 0 & 0 & 0 & 1 }. \]
\medskip
\[ \Mtrans(s,\FQ)=\bordermatrix{
& 4 & 31 & 22 & 211& 13 & 121 & 112 & 1111 \cr
4&   1 & 0 & 0 & 0 & 0 & 0 & 0 & 0 \cr
31&   0 & 1 & 1 & 0 & 1 & 0 & 0 & 0 \cr
22&   0 & 0 & 1 & 0 & 0 & 1 & 0 & 0 \cr
211&   0 & 0 & 0 & 1 & 0 & 1 & 1 & 0 \cr
1111&   0 & 0 & 0 & 0 & 0 & 0 & 0 & 1 }. \]
\medskip
\[ \Mtrans(m,M)=\bordermatrix{
& 4 & 31 & 22 & 211& 13 & 121 & 112 & 1111 \cr
4&   1 & 0 & 0 & 0 & 0 & 0 & 0 & 0 \cr
31&   0 & 1 & 0 & 0 & 1 & 0 & 0 & 0 \cr
22&   0 & 0 & 1 & 0 & 0 & 0 & 0 & 0 \cr
211&   0 & 0 & 0 & 1 & 0 & 1 & 1 & 0 \cr
1111&   0 & 0 & 0 & 0 & 0 & 0 & 0 & 1 }. \]
\medskip
{\small
\[ \Mtrans(\FQ,M)=\bordermatrix{ 
& 4 & 31 & 22 & 211& 13 & 121 & 112 & 1111 \cr
4&   1 & 1 & 1 & 1 & 1 & 1 & 1 & 1 \cr
31&   0 & 1 & 0 & 1 & 0 & 1 & 0 & 1 \cr
22&   0 & 0 & 1 & 1 & 0 & 0 & 1 & 1 \cr
211&   0 & 0 & 0 & 1 & 0 & 0 & 0 & 1 \cr
13&   0 & 0 & 0 & 0 & 1 & 1 & 1 & 1 \cr 
121&   0 & 0 & 0 & 0 & 0 & 1 & 0 & 1 \cr
112&   0 & 0 & 0 & 0 & 0 & 0 & 1 & 1 \cr 
1111&   0 & 0 & 0 & 0 & 0 & 0 & 0 & 1 }. \]
}
\medskip
{\small
\[ \Mtrans(M,\FQ)=\bordermatrix{ 
& 4 & 31 & 22 & 211& 13 & 121 & 112 & 1111 \cr
4&   1 & -1 & -1 & 1 & -1 & 1 & 1 & -1 \cr
31&   0 & 1 & 0 & -1 & 0 & -1 & 0 & 1 \cr
22&   0 & 0 & 1 & -1 & 0 & 0 & -1 & 1 \cr
211&   0 & 0 & 0 & 1 & 0 & 0 & 0 & -1 \cr
13&   0 & 0 & 0 & 0 & 1 & -1 & -1 & 1 \cr 
121&   0 & 0 & 0 & 0 & 0 & 1 & 0 & -1 \cr
112&   0 & 0 & 0 & 0 & 0 & 0 & 1 & -1 \cr 
1111&   0 & 0 & 0 & 0 & 0 & 0 & 0 & 1 }. \]
}
\medskip
{\small
\[ \Mtrans(G,\FQ) = \bordermatrix{%
& 4 & 31 & 22 & 211 & 13 & 121 & 112 & 1111 \cr 
4 &    1 & -t & -t & t^2 & -t & t^2 & t^2  &-t^3 \cr
31 &   0 &  1 &  0 & -t &  0 & -t &  0 & t^2 \cr
22 &   0 &  0 &  1 & -t^2 &  0 &  0 & -t & t^3 \cr
211 &  0 &  0 &  0 &  1 &  0 &  0 &  0 & -t \cr
13 &   0 &  0 &  0 &  0 &  1 & -t^2 & -t^2 & t^4 \cr
121 &  0 &  0 &  0 &  0 &  0 &  1 &  0 & -t^2 \cr
112 &  0 &  0 &  0 &  0 &  0 &  0 &  1 & -t^3 \cr
1111 & 0 &  0 &  0 &  0 &  0 &  0 &  0 &  1 
}. \] 
}
\medskip
{\small
\[ \Mtrans(G,M) = \bordermatrix{%
& 4 & 31 & 22 & 211 & 13 & 121 & 112 & 1111 \cr 
4 &    1 & 1-t & 1-t & (1-t)^2 & 1-t & (1-t)^2 & (1-t)^2  &(1-t)^3 \cr
31 &   0 &  1 &  0 & 1-t &  0 & 1-t &  0 & (1-t)^2 \cr
22 &   0 &  0 &  1 & 1-t^2 &  0 &  0 & 1-t & (1-t)(1-t^2) \cr
211 &  0 &  0 &  0 &  1 &  0 &  0 &  0 & 1-t \cr
13 &   0 &  0 &  0 &  0 &  1 & 1-t^2 & 1-t^2 & (1-t^2)^2 \cr
121 &  0 &  0 &  0 &  0 &  0 &  1 &  0 & 1-t^2 \cr
112 &  0 &  0 &  0 &  0 &  0 &  0 &  1 & 1-t^3 \cr
1111 & 0 &  0 &  0 &  0 &  0 &  0 &  0 &  1 
}. \] 
}
\medskip
{\small 
\[ \Mtrans(\mathcal{S},M) = \bordermatrix{%
& 4 & 31 & 22 & 211 & 13 & 121 & 112 & 1111 \cr 
4 &    1 & 1 & 1 & 1 & 1 & 1 & 1 & 1 \cr
31 &   0 & 1 & 0 & 1 & 0 & 1 & 0 & 1 \cr
22 &   0 & 0 & 1 & 1 & 0 & 1 & 1 & 2 \cr
211 &  0 & 0 & 0 & 1 & 0 & 0 & 0 & 1 \cr
13 &   0 & 0 & 1 & 1 & 1 & 1 & 2 & 2 \cr
121 &  0 & 0 & 0 & 0 & 0 & 1 & 0 & 1 \cr
112 &  0 & 0 & 0 & 0 & 0 & 0 & 1 & 1 \cr
1111 & 0 & 0 & 0 & 0 & 0 & 0 & 0 & 1
}. \]
}
\medskip 
{\small
\[ \Mtrans(\mathcal{S},F) = \bordermatrix{%
& 4 & 31 & 22 & 211 & 13 & 121 & 112 & 1111 \cr 
4 &    1 & 0 & 0 & 0 & 0 & 0 & 0 & 0 \cr
31 &   0 & 1 & 0 & 0 & 0 & 0 & 0 & 0 \cr
22 &   0 & 0 & 1 & 0 & 0 & 1 & 0 & 0 \cr
211 &  0 & 0 & 0 & 1 & 0 & 0 & 0 & 0 \cr
13 &   0 & 0 & 1 & 0 & 1 & 0 & 0 & 0 \cr
121 &  0 & 0 & 0 & 0 & 0 & 1 & 0 & 0 \cr
112 &  0 & 0 & 0 & 0 & 0 & 0 & 1 & 0 \cr
1111 & 0 & 0 & 0 & 0 & 0 & 0 & 0 & 1
}. \]
}
\medskip
{\small
\[ \Mtrans(K,F) = \bordermatrix{%
& 4 & 31 & 22 & 211 & 13 & 121 & 112 & 1111 \cr 
4 &    1 & 1 & 1 & 1 & 1 & 1 & 1 & 1 \cr
31 &   0 & 1 & 1 & 0 & 0 & 1 & 1 & 0 \cr
22 &   0 & 0 & 1 & 1 & 1 & 1 & 0 & 0
}. \]
}
\medskip
{\small
\[ \Mtrans(K,M) = \bordermatrix{%
& 4 & 31 & 22 & 211 & 13 & 121 & 112 & 1111 \cr 
4 &    1 & 2 & 2 & 4 & 2 & 4 & 4 & 8 \cr
31 &   0 & 1 & 1 & 2 & 0 & 2 & 2 & 4 \cr
22 &   0 & 0 & 1 & 2 & 1 & 2 & 2 & 4
}. \]
}
\medskip
\[ 
\Mtrans(P,\FQ) = \bordermatrix{%
& 4 & 31 & 22 & 211 & 13 & 121 & 112 & 1111 \cr 
4   & 1  & -t  & -t & t^2  & -t & t^2 & t^2 & -t^3 \cr
31  & 0 & 1 & 1-t & -t & 1 & -2t  & -t  & t^3 + t^2 \cr
22  & 0 & 0 & 1  & -t & 0 & 1-t  & -t & t^3 \cr
211 & 0 & 0 & 0 & 1 & 0 & 1 & 1 & -t^3 - t^2 - t \cr
1111 & 0 & 0 & 0 & 0 & 0 & 0 & 0 & 1
}. \] 
\medskip 
$\Mtrans(Q,\FQ)=D\cdot\Mtrans(P,\FQ)$ where $D_{5\times 5}$ 
has diagonal entries
\[1-t,\quad (1-t)^2,\quad (1-t)(1-t^2),\quad (1-t)^2(1-t^2),\quad
(1-t)(1-t^2)(1-t^3)(1-t^4). \]
\medskip
{\small
\[ \Mtrans(\FQ,G) = \bordermatrix{%
& 4 & 31 & 22 & 211 & 13 & 121 & 112 & 1111 \cr 
4 &    1 & t & t & t^3 & t & t^3 & t^3 & t^6 \cr
31 &   0 & 1 & 0 & t & 0 & t & 0 & t^3 \cr
22 &   0 & 0 & 1 & t^2 & 0 & 0 & t & t^4 \cr
211 &  0 & 0 & 0 & 1 & 0 & 0 & 0 & t \cr
13 &   0 & 0 & 0 & 0 & 1 & t^2 & t^2 & t^5 \cr
121 &  0 & 0 & 0 & 0 & 0 & 1 & 0 & t^2 \cr
112 &  0 & 0 & 0 & 0 & 0 & 0 & 1 & t^3 \cr
1111 & 0 & 0 & 0 & 0 & 0 & 0 & 0 & 1
}. \]
}
\medskip
$\Mtrans(M,G)=$
{\footnotesize
\[ 
\bordermatrix{ 
& 4 & 31 & 22 & 211& 13 & 121 & 112 & 1111 \cr
4&   1 & -(1-t) & -(1-t) & (1-t)(1-t^2) & -(1-t) & (1-t)(1-t^2) 
        & (1-t)(1-t^2) & -(1-t)(1-t^2)(1-t^3) \cr
31&   0 & 1 & 0 & -(1-t) & 0 & -(1-t) & 0 & (1-t)(1-t^2) \cr
22&   0 & 0 & 1 & -(1-t^2) & 0 & 0 & -(1-t) & (1-t)(1-t^3) \cr
211&   0 & 0 & 0 & 1 & 0 & 0 & 0 & -(1-t) \cr
13&   0 & 0 & 0 & 0 & 1 & -(1-t^2) & -(1-t^2) & (1-t^2)(1-t^3) \cr 
121&   0 & 0 & 0 & 0 & 0 & 1 & 0 & -(1-t^2) \cr
112&   0 & 0 & 0 & 0 & 0 & 0 & 1 & -(1-t^3) \cr 
1111&   0 & 0 & 0 & 0 & 0 & 0 & 0 & 1 }. \]
} 
\medskip
\[ \Mtrans(P,G) = \bordermatrix{%
& 4 & 31 & 22 & 211 & 13 & 121 & 112 & 1111 \cr 
4    & 1 &  0 &  0 & 0 & 0 & 0 & 0 & 0 \cr
31   & 0 &  1 &  -t + 1 & -t^3 + t^2 & 1 & t^2 - t & 0 &  0 \cr
22   & 0 &  0 &  1 & t^2 - t & 0  & -t + 1 & 0 & 0 \cr
211  & 0 &  0 &  0 & 1 & 0 & 1 & 1 & 0 \cr
1111 & 0 &  0 &  0 & 0 & 0 & 0 & 0 & 1
}. \]
\medskip
{\small 
\[ \Mtrans(\mathcal{S},G) = \bordermatrix{%
& 4 & 31 & 22 & 211 & 13 & 121 & 112 & 1111 \cr 
4 &    1 & t & t & t^3 & t & t^3 & t^3 & t^6 \cr
31 &   0 & 1 & 0 & t & 0 & t & 0 & t^3 \cr
22 &   0 & 0 & 1 & t^2 & 0 & 1 & t & t^2+t^4 \cr
211 &  0 & 0 & 0 & 1 & 0 & 0 & 0 & t \cr
13 &   0 & 0 & 1 & t^2 & 1 & t^2 & t+t^2 & t^4+t^5 \cr
121 &  0 & 0 & 0 & 0 & 0 & 1 & 0 & t^2 \cr
112 &  0 & 0 & 0 & 0 & 0 & 0 & 1 & t^3 \cr
1111 & 0 & 0 & 0 & 0 & 0 & 0 & 0 & 1
}. \]
}
\medskip
$\Mtrans(K,G)=$
{\footnotesize 
\[ \bordermatrix{%
& 4 & 31 & 22 & 211 & 13 & 121 & 112 & 1111 \cr 
4 &    1 & 1+t & 1+t & (1+t)(1+t^2) & 1+t & (1+t)(1+t^2) & (1+t)(1+t^2) 
& (1+t)(1+t^2)(1+t^3) \cr
31 &   0 & 1 & 1 & t(1+t) & 0 & 1+t & 1+t & t^2(1+t)^2 \cr
22 &   0 & 0 & 1 & 1+t^2 & 1 & 1+t^2 & t(1+t) & t(1+t)(1+t^3)
}. \]
}
\medskip 
$\Mtrans(H,\FQ)=$
{\footnotesize
\[ \bordermatrix{%
& 4 & 31 & 22 & 211 & 13 & 121 & 112 & 1111 \cr 
4 & 1 & q+q^2+q^3 & q+2q^2+q^3+q^4 & q^3+q^4+q^5 & q+q^2+q^3
  & q^2+q^3+2q^4+q^5 & q^3+q^4+q^5 & q^6 \cr
31 & t & 1+qt+q^2t & 1+q+qt+2q^2t & q+q^2+q^3t & 1+qt+q^2t 
  & 2q+q^2+q^2t+q^3t & q+q^2+q^3t & q^3 \cr
22 & t^2 & t+qt+qt^2 & 1+t+qt+qt^2+q^2t^2 & q+qt+q^2t 
  & t+qt+qt^2 & 1+q+qt+q^2t+q^2t^2 & q+qt+q^2t & q^2 \cr
211 & t^3 & t+t^2+qt^3 & 2t+t^2+qt^2+qt^3 & 1+qt+qt^2 
   & t+t^2+qt^3 & 1+t+qt+2qt^2 & 1+qt+qt^2 & q \cr
1111 & t^6 & t^3+t^4+t^5 & t^2+t^3+2t^4+t^5 & t+t^2+t^3
  & t^3+t^4+t^5 & t+2t^2+t^3+t^4 & t+t^2+t^3 & 1
}. \]
}
\medskip
$\Mtrans(Q',\FQ)=$
{\footnotesize
\[ \bordermatrix{%
& 4 & 31 & 22 & 211 & 13 & 121 & 112 & 1111 \cr 
4 & 1 & 0 & 0 & 0 & 0 & 0 & 0 & 0 \cr 
31 & t & 1 & 1 & 0 & 1 & 0 & 0 & 0 \cr
22 & t^2 & t & 1+t & 0 & t & 1 & 0 & 0 \cr
211 & t^3 & t+t^2 & 2t+t^2 & 1 & t+t^2 & 1+t & 1 & 0 \cr
1111 & t^6 & t^3+t^4+t^5 & t^2+t^3+2t^4+t^5 & t+t^2+t^3
  & t^3+t^4+t^5 & t+2t^2+t^3+t^4 & t+t^2+t^3 & 1
}. \]
}

\section{Acknowledgments} We thank the anonymous referees for very
helpful comments.

\bibliographystyle{plain}
\bibliography{hltrans}

\begin{thebibliography}{10}

\bibitem{bbssz}
C.~Berg, N.~Bergeron, F.~Saliola, Serrano L., and M.~Zabrocki.
\newblock A lift of the {S}chur and {H}all-{L}ittlewood bases to
  non-commutative symmetric functions.
\newblock {\em Canadian J. Math}, To appear.

\bibitem{bz}
N.~Bergeron and M.~Zabrocki.
\newblock $q$ and $q,t$-analogs of non-commutative symmetric functions.
\newblock {\em Discrete Math}, 298:79--103, 2005.

\bibitem{butler}
Lynne Butler.
\newblock Subgroup lattices and symmetric functions.
\newblock {\em Mem. Amer. Math. Soc.}, 112(539), 1994.

\bibitem{carbonara}
Joaquin~O. Carbonara.
\newblock A combinatorial interpretation of the inverse {$t$}-{K}ostka matrix.
\newblock {\em Discrete Math.}, 193(1-3):117--145, 1998.
\newblock Selected papers in honor of Adriano Garsia (Taormina, 1994).

\bibitem{DLT-HL-SLC}
J.~D\'esarm\'enien, B.~Leclerc, and J.-Y. Thibon.
\newblock {H}all-{L}ittlewood functions and {K}ostka-{F}oulkes polynomials in
  representation theory.
\newblock {\em S\'em. Lothar. Combin.}, 32:Art. B32c, approx. 38 pages, 1994.

\bibitem{elw}
E.~Egge, N.~Loehr, and G.~Warrington.
\newblock From quasisymmetric expansions to {S}chur expansions via a modified
  inverse {K}ostka matrix.
\newblock {\em European J. Combin.}, 31(8):2014--2027, 2010.

\bibitem{ER-kinv}
\"{O}. E\u{g}ecio\u{g}lu and J.~Remmel.
\newblock A combinatorial interpretation of the inverse {K}ostka matrix.
\newblock {\em Linear Multilinear Algebra}, 26:59--84, 1990.

\bibitem{garsia-orthog-milne}
Adriano Garsia.
\newblock Orthogonality of {M}ilne's polynomials and raising operators.
\newblock {\em Discrete Math.}, 99:247--264, 1992.

\bibitem{gkllrt}
Israel~M. Gelfand, Daniel Krob, Alain Lascoux, Bernard Leclerc, Vladimir~S.
  Retakh, and Jean-Yves Thibon.
\newblock Noncommutative symmetric functions.
\newblock {\em Adv. Math.}, 112(2):218--348, 1995.

\bibitem{gessel}
Ira~M. Gessel.
\newblock Multipartite {$P$}-partitions and inner products of skew {S}chur
  functions.
\newblock In {\em Combinatorics and algebra ({B}oulder, {C}olo., 1983)},
  volume~34 of {\em Contemp. Math.}, pages 289--317. Amer. Math. Soc.,
  Providence, RI, 1984.

\bibitem{hag}
J.~Haglund.
\newblock A combinatorial model for the {M}acdonald polynomials.
\newblock {\em Proc. Natl. Acad. Sci. USA}, 101(46):16127--16131 (electronic),
  2004.

\bibitem{HHL}
J.~Haglund, M.~Haiman, and N.~Loehr.
\newblock A combinatorial formula for {M}acdonald polynomials.
\newblock {\em J. Amer. Math. Soc.}, 102:2690--2696, 2005.

\bibitem{hlmvw}
J.~Haglund, K.~Luoto, S.~Mason, and S.~van Willigenburg.
\newblock Quasisymmetric {S}chur functions.
\newblock {\em J. Combin. Theory Ser. A}, 118(2):463--490, 2011.

\bibitem{Hai}
M.~D. Haiman.
\newblock On mixed insertion, symmetry, and shifted {Y}oung tableaux.
\newblock {\em J. Combin. Theory Ser. A}, 50:196--225, 1989.

\bibitem{hivert}
Florent Hivert.
\newblock Hecke algebras, difference operators, and quasi-symmetric functions.
\newblock {\em Adv. Math.}, 155(2):181--238, 2000.

\bibitem{kirillov-HL}
Anatol Kirillov.
\newblock New combinatorial formula for modified {H}all-{L}ittlewood
  polynomials (in $q$-series from a contemporary perspective).
\newblock {\em Contemp. Math.}, 254:283--333, 2000.

\bibitem{LNT}
A.~Lascoux, J.-C. Novelli, and J.-Y. Thibon.
\newblock Noncommutative symmetric functions with matrix parameters.
\newblock {\em J. Algebraic Combin.}, 37:621--642, 2013.

\bibitem{LS-chg}
A.~Lascoux and M.-P. Sch\"utzenberger.
\newblock Sur une conjecture de {H}. {O}. {F}oulkes.
\newblock {\em C. R. Acad. Sci. Paris S\'er. A-B}, 286A:323--A324, 1978.

\bibitem{llt}
Alain Lascoux, Bernard Leclerc, and Jean-Yves Thibon.
\newblock Ribbon tableaux, {H}all-{L}ittlewood functions and unipotent
  varieties.
\newblock {\em S\'em. Lothar. Combin.}, 34:Art.\ B34g, approx.\ 23 pp.\
  (electronic), 1995.

\bibitem{littlew}
D.~E. Littlewood.
\newblock On certain symmetric functions.
\newblock {\em Proc. London Math. Soc. (3)}, 11:485--498, 1961.

\bibitem{nablaschur}
N.~Loehr and G.~Warrington.
\newblock Nested quantum {D}yck paths and {$\nabla(s_\lambda)$}.
\newblock {\em Intl. Math. Research Notices}, 2008(5):article ID rnm157, 29
  pages, 2008.

\bibitem{spleth}
N.~Loehr and G.~Warrington.
\newblock Quasisymmetric expansions of {S}chur-function plethysms.
\newblock {\em Proc. of Amer. Math. Soc.}, 140:1159--1171, 2012.

\bibitem{Macd}
I.~G. Macdonald.
\newblock {\em Symmetric functions and {H}all polynomials}.
\newblock Oxford Mathematical Monographs. Oxford University Press, New York,
  second edition, 1995.
\newblock With contributions by A. Zelevinsky, Oxford Science Publications.

\bibitem{ntw}
J.-C. Novelli, J.-Y. Thibon, and L.~K. Williams.
\newblock Combinatorial {H}opf algebras, noncommutative {H}all-{L}ittlewood
  functions, and permutation tableaux.
\newblock {\em Adv. Math.}, 224(4):1311--1348, 2010.

\bibitem{Sag}
B.~E. Sagan.
\newblock Shifted tableaux, {S}chur {$Q$}-functions, and a conjecture of {R}.
  {P}. stanley.
\newblock {\em J. Combin. Theory Ser. A}, 45:62--103, 1987.

\bibitem{sage}
W.\thinspace{}A. Stein et~al.
\newblock {\em {S}age {M}athematics {S}oftware ({V}ersion 4.7)}.
\newblock The Sage Development Team, 2009.
\newblock {\tt http://www.sagemath.org}.

\bibitem{Ste1}
J.~R. Stembridge.
\newblock Shifted tableaux and the projective representations of symmetric
  groups.
\newblock {\em Adv. Math.}, 74:87--134, 1989.

\bibitem{Ste2}
J.~R. Stembridge.
\newblock Enriched {$P$}-partitions.
\newblock {\em Trans. Amer. Math. Soc.}, 349:763--788, 1997.

\bibitem{tevlin}
L.~Tevlin.
\newblock Noncommutative analogs of monomial symmetric functions, {Cauchy}
  identity, and {Hall} scalar product.
\newblock \texttt{arXiv:0712.2201}, 2007.

\bibitem{tevlin11}
Lenny Tevlin.
\newblock Noncommutative symmetric {H}all-{L}ittlewood polynomials.
\newblock In {\em 23rd {I}nternational {C}onference on {F}ormal {P}ower
  {S}eries and {A}lgebraic {C}ombinatorics ({FPSAC} 2011)}, Discrete Math.
  Theor. Comput. Sci. Proc., AO, pages 915--925. Assoc. Discrete Math. Theor.
  Comput. Sci., Nancy, 2011.

\bibitem{warnaar-HL}
S.~Ole Warnaar and Wadim Zudilin.
\newblock Dedekind's $\eta$-function and {R}ogers-{R}amanujan identities.
\newblock {\em Bull. Lond. Math. Soc.}, 44:1--11, 2012.

\bibitem{sage-code}
Gregory~S. Warrington.
\newblock Sage worksheet for {H}all-{L}ittlewood transition matrices (available
  online), 2013.

\end{thebibliography}

\end{document}